\documentclass[10pt]{article}

\usepackage{latexsym,amssymb,amsmath,amsthm,paralist,a4,tocloft}
\usepackage{mathrsfs}

\usepackage[T1]{fontenc}
\usepackage[sc]{mathpazo}
\usepackage[all]{xy}

\newcommand{\Spec}{\text{\it Spec}}

\newcommand{\Aut}{\text{\it Aut}}

\newcommand{\Mor}{\text{\it Mor}}

\newcommand{\Sch}{\text{\it Sch}}
\newcommand{\Q}{{\mathbb Q}}
\newcommand{\Z}{{\mathbb Z}}
\newcommand{\F}{{\mathbb F}}

\renewcommand{\O}{{\mathcal{O}}}
\newcommand{\ab}{\text{\it ab}}

\newcommand{\p}{{\mathfrak p}}

\newcommand{\sm}{{\,\smallsetminus\,}}

\newcommand{\Gal}{\text{\it Gal}}

\newcommand{\freeproductmed}{\mathop{\lower.2mm\hbox{\emas \symbol{3}}}\limits}
\newcommand{\lang}{\longrightarrow}

\newcommand{\tr}{\mathit{tr}}

\newcommand{\Val}{\mathrm{Val}}
\newcommand{\m}{\mathfrak{m}}
\newcommand{\Pe}{\mathbb{P}}
\newcommand{\A}{\mathbb{A}}
\newcommand{\sep}{\mathit{sep}}
\newcommand{\red}{\mathit{red}}

\newtheoremstyle{alex}
  {}
  {}
  {\it}
  {}
  {\bf}
  {.}
  {.5em}
  {}
\newtheoremstyle{alexdef}
  {}
  {}
  {\rm }
  {}
  {\bf}
  {.}
  {.5em}
  {}
\theoremstyle{alex}
\newtheorem{theorem}{Theorem}[section]

\newtheorem{lemma}[theorem]{Lemma}
\newtheorem{proposition}[theorem]{Proposition}

\theoremstyle{alexdef}
\newtheorem*{definition}{Definition}

\newtheorem{remark}[theorem]{Remark}

\title{\bf On different notions of tameness in arithmetic geometry}
\author{Moritz Kerz and Alexander Schmidt}
\date{March 17, 2009}
\begin{document}
\abovedisplayskip=3pt plus 1pt minus 1pt
\abovedisplayshortskip=-4pt plus 1pt
\belowdisplayskip=2pt plus 1pt minus 1pt
\belowdisplayshortskip=2pt plus 1pt minus 1pt
\maketitle

\begin{quote}
{\em Abstract:}  The notion of a tamely ramified covering is canonical only for curves. Several notions of tameness for coverings of higher dimensional schemes have been used in the literature. We show that all these definitions are essentially equivalent. Furthermore, we prove finiteness theorems for the tame fundamental groups of arithmetic schemes.
\end{quote}

\bigskip

\section{Introduction}

Let $\bar C$ be a be a proper, connected and regular curve (i.e.\ $\dim C=1$) of finite type over $\Spec(\Z)$ and let $C\subset \bar C$ be a nonempty open subscheme. Every point $x\in \bar C \sm C$ defines a discrete rank one valuation $v_x$ on the function field $k(C)$.  One says that an \'{e}tale covering $C'\to C$ is {\em tamely ramified along $\bar C \sm C$} if for every $x\in \bar C \sm C$ the valuation $v_x$ is tamely ramified in $k(C')|k(C)$.
Since the proper, regular curve $\bar C$ is determined by $C$, we can say that the \'{e}tale covering $C'\to C$ is {\em tame}  if it is tamely ramified along $\bar C \sm C$.

\smallskip
Following \cite{sga1,G-M}, one possible extension of this definition to higher dimensions is the following. We denote by $\Sch(\Z)$ the category of separated schemes of finite type over $\Spec(\Z)$.  Let $X\in \Sch(\Z)$ be a regular scheme together with an open embedding into a regular, proper scheme $\bar X \in \Sch(\Z)$ such that $\bar X \sm X$ is a normal crossing divisor (NCD) on $\bar X$. Then an \'{e}tale covering $Y \to X$ is called {\em tamely ramified along $\bar X \sm X$} if the discrete valuations associated with the generic points of $\bar X \sm X$ are tamely ramified in $k(Y)|k(X)$. However, there might exist many or (at our present knowledge about resolution of singularities) even no regular compactifications $\bar X$ of $X$ such that $\bar X\sm X$ is a NCD. Furthermore, there is no obvious functoriality for the tame fundamental group.

\smallskip
So the question for a good notion of tameness of an \'{e}tale covering $Y\to X$ of regular schemes in $\Sch(\Z)$ naturally occurs.  In this paper we compare several possible definitions:

\medskip
\begin{compactitem}
\item[\bf curve-tameness:] for every morphism $C\to X$ with $C\in \Sch(\Z)$ a regular curve, the base change $Y\times_X C \to C$ is tame.
\item[{\bf divisor-tameness:}] for every normal compactification $\bar X$ of $X$ and every point $x\in \bar X \sm X$ with $\text{codim}_{\bar X} x=1$, the discrete rank one valuation $v_x$ on $k(X)$ associated with $x$ is tamely ramified in $k(Y)|k(X)$.
\item[\bf chain-tameness:] There exists a compactification $\bar X$ of $X$ such that every discrete valuation of rank $d=\dim X$ on $k(X)$ which dominates a Parshin chain on $\bar X$  is tamely ramified in $k(Y)|k(X)$.
\item[\bf valuation-tameness:] every nonarchimedean valuation of $k(X)$ with center outside $X$ is tamely ramified in $k(Y)|k(X)$.
\end{compactitem}

\smallskip\noindent
The notion of curve-tameness has been considered in~\cite{W-tame} and  the notion of chain-tameness in~\cite{S-tame}.
Curve-tameness is the maximal definition of tameness which is stable under base change and extends the given definition for curves. Valuation tameness is obviously stronger than divisor-tameness and chain-tameness.  Our first result is the following:

\begin{theorem}[see Theorem~\ref{triangle-theo}]\label{main}
The notions of curve-tameness, divisor-tameness and chain-tameness   are equivalent. If\/ every intermediate field between $k(X)$ and the Galois closure of $k(Y)$ over $k(X)$ admits a regular proper model, then they are equivalent to valuation-tameness.

Moreover, if there exists a regular compactification $\bar X$ such that $\bar X \sm X$ is a NCD, then all these notions of tameness are equivalent to the notion of tame ramification along $\bar X \sm X$ of\/ \cite{sga1,G-M}.
\end{theorem}

Suppose that $\pi: Y \to X$ is Galois with group $G$ and that we are given a fixed normal compactification $\bar X$ of $X$. Denoting the normalization of $\bar X$ in $k(Y)$ by $\bar Y$, there are the following additional notions of tameness:

\medskip
\begin{compactitem}
\item[{\bf numerical tameness:}] for every $y\in \bar Y \sm Y$ the inertia group $T_y(\bar Y|\bar X)\subset G$ is of order prime to the characteristic of the residue field $k(y)$.
\item[\bf cohomological tameness:] for every $x\in \bar X \sm X$ the semi-local ring $\O_{\bar Y, \pi^{-1}(x)}$ is a cohomologically trivial $G$-module.
\end{compactitem}

\smallskip\noindent
The notion of numerical tameness has been considered in \cite{Ab}, \cite{CE}, \cite{S-tame} and \cite{W-tame}.
Cohomological tameness has been considered in \cite{CEPT} (in the more general context of group scheme actions).

\medskip
Our second result is :

\begin{theorem}[see Theorems~\ref{ntimpliesvaltame} and \ref{cohotame}] \label{main2}
Numerical tameness and cohomological tameness are equivalent and imply valuation tameness. If $\bar X \sm X$ is a NCD or $G$ is nilpotent, then all definitions are equivalent.
\end{theorem}

We show Theorems~\ref{main} and \ref{main2} in the more general context of regular schemes, separated and of finite type over an integral, pure-dimensional and excellent base scheme $S$. In the appendix we give illustrating examples, in particular, an example of an \'etale covering of a regular scheme which is numerically tame with respect to one regular compactification (hence curve-, divisor- and chain-tame) but not numerically tame with respect to another regular compactification.

\medskip
In the case that the function field of $S$ is absolutely finitely generated, we show the following tame variant of a finiteness theorem  of Katz and Lang \cite{K-L}.

\begin{theorem}[see Theorem~\ref{Katz-Lang}] Let
$S$ be an integral, pure-dimensional and excellent base scheme whose function field is absolutely finitely generated.  Let $f:Y\to X$ be a smooth, surjective morphism of connected regular schemes which are separated and of finite type over $S$. Assume that the generic
fibre of $f$ is geometrically connected and that one of the following conditions {\rm (i)} and {\rm (ii)} is satisfied:

\medskip

\begin{compactitem}
\item[\rm (i)]  $X$ has a regular compactification $\bar X$ over $S$ such that the boundary $\bar X \sm X$ is a normal crossing divisor.\smallskip
\item[\rm (ii)] The generic fibre of $f$ has a rational point.
\end{compactitem}

\medskip\noindent
Then the group
\[
\ker_S^t (Y/X):=\ker \big( \pi_1^{t,\ab}(Y/S) \to  \pi_1^{t,\ab}(X/S) \big)
\]
is finite.
\end{theorem}

\medskip
The authors thank A.~Holschbach and F.~Pop for helpful discussions on the subject, and T.~Chinburg for his proposal to consider cohomological tameness.

\section{The Key Lemma}

We start by recalling the following desingularization results. See \cite{Lip} for a proof of (i) and \cite{saf}, Lecture 3, Theorem on p.\,38 and Remark~2 on p.\,43, for (ii). We denote the set of regular points of a scheme $X$ by $X^{reg}$.

\begin{proposition}\label{desing} Let $X$ be a two-dimensional,  normal, connected and excellent scheme. Then the following hold.

\smallskip
\begin{compactitem}
\item[\rm (i)] After a finite number of blow-ups in closed points followed by normalizations, we obtain a proper morphism $p: X'\to X$ with $X'$ regular and such that $p^{-1}(X^{reg}) \to X^{reg}$ is an isomorphism of schemes. \smallskip
\item[\rm (ii)] Assume that $X$ is regular and let $Z$ be a proper closed subset on $X$. After a finite number of blow-ups in closed points, we obtain a proper morphism $p: X'\to X$  such that $p^{-1}(Z)$ is a strict normal crossing divisor on $X'$.
\end{compactitem}
\end{proposition}

Let $X$ be a  normal, noetherian scheme and let $X'\subset X$ be a dense open subscheme. Assume we are given an an \'{e}tale covering $Y'\to X'$.

\begin{definition} Let $x\in X\sm X'$ be a point. We say that $Y'\to X'$ is {\em unramified along $x$} if it extends to an \'{e}tale covering of some open subscheme $U\subset X$ which contains $X'$ and $x$. Otherwise we say that $Y'\to X'$ {\em ramifies along $x$}.
\end{definition}

\begin{remark}
Let $Y$ be the normalization of $X$ in $k(Y')$. Then the branch locus of $Y\to X$ consists of all points $x\in X\sm X'$ such that $Y'\to X'$ ramifies along $x$.
\end{remark}

\begin{remark}\label{tamealongremark}
If $\mathrm{codim}_X \{x\} =1$,  then $Y'\to X'$  ramifies along $x$ if and only if the discrete valuation of $k(X')$ associated with $x$ ramifies in $k(Y')$. In this case we can speak about {\em tame} and {\em wild} ramification along $x$ by referring to the associated valuation.
\end{remark}

The Key Lemma for our investigations is the following Lemma~\ref{key-lemma}.
An integral, one-dimensional scheme will be called a curve.
We denote the normalization of a curve $C$ in its function field $k(C)$ by $\tilde{C}$.

\begin{lemma}[Key Lemma]\label{key-lemma} Let $A$ be a local, normal and excellent ring and let
$X'\subset X=\Spec(A)$ be a nonempty open subscheme. Let $Y'\to X'$ be an \'{e}tale Galois covering of prime degree $p$.  Assume that $X\sm X'$ contains an irreducible component $D$ of codimension one in $X$ such that $Y'\to X'$ is ramified along the generic point of\/ $D$. Then there exists a curve $C$ on $X$ with $C':=C \cap X'\ne\varnothing$ such that the base change $Y'\times_{X'} \tilde C' \to \tilde C'$ is ramified along a point of\/ $\tilde C \sm \tilde C'$.

\end{lemma}

In the proof of the Key Lemma we use a reduction argument to dimension two which is due to Wiesend~\cite[Proof of Thm.~2]{W-tame}. Wiesend's arguments
for the two-dimensional case are however incomplete.
O.~Gabber has obtained a similar result using a slightly different technique.

In the proof of our Key Lemma we need the following ramification criterion.

\begin{lemma}\label{ram-crit} Let $R$ be a discrete valuation ring with prime element $\pi$, $K$ the quotient field of $R$ and $p$  a prime number.

\smallskip
\begin{compactitem}
\item[\rm (i)] If $\/ K'|K$ is a separable extension of discretely valued fields of degree $p$ with trivial residue field extension and ramification index $p$, then the minimal polynomial $f$ of any prime element of\/ $R'$ is of the form
 $f=T^p + a_{p-1} T^{p-1} + \cdots + a_0\in R[T]$ with $\pi| a_i$ for
every $i$ and  $\pi^2 \nmid a_0$.\smallskip
\item[\rm (ii)] Let $f\in K[T]$ be a separable polynomial of the form $f=T^p + a_{p-1} T^{p-1}+ \cdots  + a_0$ with $a_0\neq 0$, $p\nmid v_R(a_0)$, and such that
\[
\min_{i\ge 0} v_R(a_i) > p\,  \max_{i\ge 0} [v_R(a_0)-v_R(a_i)]\; .
\]
Then $K'=K[T]/(f)$ is a discretely valued field with ramification index $p$ over $K$.
\end{compactitem}
\end{lemma}

\begin{proof} Assertion {\rm (i)} is standard, see for example \cite[I.6 Proposition 18]{Ser-loc}. For assertion {\rm (ii)}, we substitute $T$ by $\pi^l T$, where $l=\max_{i\ge 0} [v_R(a_0)-v_R(a_i)]$, and divide $f$ by $\pi^{p \, l}$. Let
\[
g=T^p + b_{p-1} T^{p-1}+ \cdots  + b_0
\]
denote the resulting monic polynomial. We claim that $0< v_R(b_0) \le v_R(b_i)$ for all $i\ge 0$ and still $p\nmid v_R(b_0)$. In fact, we have  $b_i=\pi^{l(i-p)} a_i$ so that $ v_R(b_i) = v_R(a_i)-l(p-i)> 0$.
We have $ l\ge v_R(a_0)-v_R(a_i) $ by definition, which implies
\[
v_R(b_i) - v_R(b_0) = v_R(a_i) - v_R(a_0)+ l p -l(p-i)= v_R(a_i) - v_R(a_0) +l i \ge 0  ,
\]
validating our claim.

Let $R'$ be the normalization of $R$ in $K[T]/(g)\cong K[T]/(f)=K'$.
Let $R''$ be the localization of $R'$
at some maximal ideal.
We will show that the ramification index $e$ of $R''|R$ is $p$. This shows
that $R'=R''$ and completes the proof of (ii). Let $t\in R''$ be the
image of the variable $T$ under the homomorphism $R[T]/(g)\to R''$. The equation
\[
t^p=-b_{p-1} t^{p-1} - \cdots - b_0
\]
shows that $t\notin R''^\times$. Furthermore, it gives the equality in the middle of
$$p\, v_{R''}(t) =v_{R''}(t^p)=v_{R''}(b_0)=e\, v_{R}(b_0)\; .$$
 As $p\nmid v_{R}(b_0)$,
this implies $p|e$ and, as $e\le p$, this indeed means $e=p$.
\end{proof}

\begin{proof}[Proof of the Key Lemma~\ref{key-lemma}.]
Let $Y$ be the normalization of $X$ in $k(Y')$.  Then
$Y=\Spec (B)$ for some semi-local ring $B$.

\smallskip
{\em Reduction to $\dim(A)=2$.} The assertion is trivial if $\dim(A)=1$. Assume that we have shown Lemma~\ref{key-lemma} for $\dim(A)=2$. We use induction on $\dim(A)$ to prove the general case. For $\dim(A)>2$, choose a point
$x_1\in D$ of dimension one. Set $X_1=\Spec(A_{x_1})$, where $A_{x_1}$ means the localization of $A$ at the prime ideal corresponding to $x_1$, and $X'_1= X_1 \cap X'$.
We deduce by our induction assumption that there exists a curve $C_1$ on $X_1$ with $C'_1=C_1\cap X'_1\ne \varnothing$
such that $Y'\times_{X'} \tilde C'_1 \to \tilde C'_1$ is ramified along $D\times_X \tilde C_1$. Let $X_2$ be the normalization of the closure of $C_1$ in $X$, $X'_2$ be $X_2\times_X X'$
and $Y'_2$ be the normalization of $X'_2$ in $Y'\otimes k(C)$. Finally, the two-dimensional case
produces a curve $C_2$ on $X_2$ with $C'_2=C_2\cap X'_2\ne \varnothing$ such that $Y'_2\times_{X'_2}\tilde C'_2\to \tilde C'_2$ is ramified along $D\times \tilde C_2$. Let $C$ be the image of $C_2$ under the morphism $X_2\to X$, $C$ is a curve on $X$ with $C'=C\cap X'\ne \varnothing$ such that $Y'\times_{X'}\tilde C' \to \tilde C'$ is ramified along $D\cap C$.
The latter since $Y\otimes_X k(C)$ is ramified
over some point of $D\times \tilde{ C}$ if $Y\otimes_X k(C_2) = Y_2 \otimes_{X_2} k(C_2)$ is ramified over some point of $D\times \tilde{C}_2$ by the base change invariance of \'{e}tale morphisms.

\smallskip
{\em Reduction to $\dim(A)=2$ and $A$ and $D$ regular.}
Applying  Proposition~\ref{desing} to  $\Spec(A)$ and $D$, and localizing at a closed point of the
strict transform of $D$, we can assume without loss generality that $A$ and $D$ are regular.

\medskip
Now Lemma~\ref{key-lemma} will be proved by distinguishing three possible cases with different ramification indices.
Since $A$ is factorial, there exists an irreducible $\pi \in A$ with $D=V(\pi)$. Now $R_D=A/(\pi)$,
 $R=A_{(\pi)}$ and $R'=B_{(\pi)}$ are discrete valuation rings. We distinguish several possible cases:

\medskip
{\em 1st case:} $v_{R'}(\pi)=p$.

\smallskip
\noindent
Let $\pi'$ be a prime element of $R'$. After multiplication by units of $R'$ we can assume that $\pi'\in A'$.
The minimal polynomial $f\in K[T]$ of $\pi'$ lies in $A[T]$ and, according to Lemma~\ref{ram-crit}\,(i),
is of the form $f=T^p + a_{p-1} T^{p-1} + \cdots + a_0$ with $\pi| a_i$ for every~$i$ but $\pi^2 \nmid a_0$. Observe that  $A[T]/(f)$ and $B$ are isomorphic over some dense open subscheme of $X$. We use Proposition~\ref{desing} to resolve singularities of the divisor generated by the coefficients of $f$ and localize at the closed point of the strict transform of $D$. This new ring will still be denoted by $A$,  and the induced open subscheme of $X=\Spec(A)$ by $X'$.

Let $\pi$ be a prime element in the new desingularized ring with $D=V(\pi)$ and denote by $t\in A$ an element such that $V(t)$ is the exceptional divisor.  Then $(t,\pi)$ is the maximal ideal of $A$ and
the $a_i$ are supported on $V(t)\cup D$, so that $a_i$ is up to a unit of $A$ equal to $t^{l_i} \pi^{e_i}$ with $e_0=1$ and $e_i\ge 1$ for $i>0$. In order to prove the Key Lemma~\ref{key-lemma}, we will construct a regular curve $C$ on $X$ of the form $C=V(t^N-\pi)$ with the required properties.

\smallskip\noindent
For this we choose a natural number $N$ such that

\begin{compactitem}
\item{$N>p\, \max_{i\ge 0}(l_0-l_i)$.}\smallskip
\item{$C=V(t^N-\pi)$ meets $X'$ and over the generic point of $C$ the algebras $A[T]/(f)$ and $B$ are isomorphic.}\smallskip
\item{$N+l_0$ is not divisible by $p$.}
\end{compactitem}
Then $R_C=A/(t^N-\pi)$ is a discrete valuation ring and, since $(t,t^N-\pi)=(t,\pi)$ is the maximal ideal of $A$, the element $t$ induces a prime element of $R_C$.
Let $K_C=Q(R_C)$. Our assumption on $N$ shows that $K'_C:=B\otimes_A K_C \cong K_C[T]/(f_C)$,
where $f_C\in R_C[T]$ is the polynomial induced by $f$. The polynomial $f_C$ is separable as $B$ is \'{e}tale over $X'$ and of the form
\[
f_C=T^p + a^C_{p-1} T^{p-1} +\cdots + a^C_0
\]
with $v_{R_C}(a_i)=N\, e_i +l_i$ for $i\ge 0$.
Therefore the assumptions  of Lemma~\ref{ram-crit}\,(ii) are fulfilled, so that $K'_C$ is discretely valued and has ramification index $p$ over $K_C$.

\medskip
{\em 2nd case:} $v_{R'}(\pi)=1$.

\smallskip\noindent
In this case the residue field extension $K'_D|K_D$ of $R'|R$ is an inseparable extension of degree $p$. In particular,
$K_D=Q(R_D)$ has characteristic $p$. Let $R'_D$ be the normalization of $R_D$ in $K'_D$ and let $\pi_D$ be a
prime element of $R_D$. Then $R'_D$ is a discrete valuation ring by \cite[VI, 8.6 Corollary 2]{Bour-comm}. There are two subcases:

\medskip
{\em Subcase 2a:} $v_{R'_D}(\pi_D)=p$.

\smallskip\noindent
In this subcase we have $K'_D=K_D(\sqrt[p]{\pi_D})$ if we choose the prime element $\pi_D$ appropriately.
Let $f\in R[T]$ be the minimal polynomial of a lift to $R'$ of $\sqrt[p]{\pi}$. After a resolution of
singularities as in Proposition \ref{desing} of the divisor generated by $D$ and the coefficients of $f$
and after localizing at the closed point of the strict transform of $D$, we can assume the following:

\smallskip

\begin{compactitem}
\item{The coefficients of $f$ are supported on $E \cup D$, where $E$ is the exceptional divisor,}
\item{$E$ and $D$ intersect transversely.}
\end{compactitem}

\smallskip\noindent
The constant coefficient $t$ of $f$ still induces a prime element in $R_D$ and by assumption $t$ is supported on $E$ so that we conclude $E=V(t)$.
So $f$ is  of the form $f=T^p + a_{p-1} T^{p-1} +  \cdots a_1 T + t$ with
$a_i$ for $i>0$ up to a unit of $A$ equal to $\pi^{e_i} t^{l_i}$ with $e_i>0$ and $l_i\in\mathbb{Z}$.
The ideal $(t,\pi)$ of $A$ is the maximal ideal, since $E$ and $D$ intersect tranversely.
There exists some dense open subscheme of $\Spec(A_t)$ such that
$B$ and $A_t[T]/(f)$ are isomorphic over this subscheme. We will construct a natural number $N$ such that the regular curve $C=V(t^N-\pi)$
has the required properties.
In fact, choose $N>0$ such that the following properties are satisfied:

\smallskip
\begin{compactitem}
\item{$N+l_i>0$ for all $i>0$.}
\item{The generic point of $V(t^N-\pi)$ lies in $X'$ and in the open subscheme of $\Spec(A_t)$ where $B$ and $A_t[T]/(f)$ are isomorphic. }
\end{compactitem}

\smallskip\noindent
Since $(t,t^N-\pi)=(t,\pi)$ is the maximal ideal of $A$, the element $t$ induces a prime element of the discrete valuation ring $R_C=A/(t^N-\pi)$.  With $K_C=Q(R_C)$,  we have
$B\otimes_A K_C \cong K_C[T] /(f_C)$, where $f_C$ is the separable polynomial in $K_C[T]$ induced by $f$. Then $f_C$ is of the form $$f_C=T^p  + a_{p-1}^C T^{p-1} + \cdots + a_0^C$$ with $v_{R_C}(a_i^C)>0$ for every $i$ and
$v_{R_C}(a_0^C)=1$. Finally, Lemma \ref{ram-crit}\,(ii) shows that $K'_C$ is discretely valued and
ramified over $K_C$ with ramification index $p$.

\medskip
{\em Subcase 2b:} $v_{R'_D}(\pi_D)=1$.

\smallskip\noindent
In this subcase the residue extension of $R'_D|R_D$ is an inseparable extension of degree $p$, so that $K'_D=K_D(\sqrt[p]{\xi})$ for some
$\xi\in R_D^\times$ which is not a $p$-th power in the residue field of $R_D$.
 Lift $\xi$ to $R'$ and denote its minimal polynomial by $f\in R[T]$. As in the
previous subcase, we resolve singularities and localize so that without loss of generality we can assume the coefficients of $f$
are supported on $E\cup D$, where $E=V(t)$ is the exceptional divisor, and so that $E$ and $D$ intersect transversely.
The prime element $t$ can be chosen arbitrarily now, in contrast to subcase 2a. Observe that the residue field of $A$ does not change in the desingularization process (we blow up a regular scheme in a regular center).
Then $f$ is  of the form $f=T^p + a_{p-1} T^{p-1} +  \cdots a_1 T + a_0$ with $a_i$ for $i>0$ up to a unit of $A$
equal to $\pi^{e_i} t^{l_i}$ with $e_i>0$ and $l_i\in\mathbb{Z}$ and such that $a_0\in A^\times$ induces an element
in the residue field of $A$ which is not a $p$-th power.
There exists an open subscheme of $\Spec(A_t)$ such that
$B$ and $A_t[T]/(f)$ are isomorphic over this subscheme. Again the curve $C$ we are searching for will be of the form $C=V(t^N-\pi)$ for some
natural number $N$. We choose $N>0$ such that the following properties are satisfied:

\smallskip
\begin{compactitem}
\item{$N+l_i>0$ for all $i>0$,}
\item{For $C=V(t^N-\pi)$ the intersection $C\cap X'$ is nonempty and the generic point of $C$ lies in the open subscheme of $\Spec ( A_t)$ over which $B$ and $A_t[T]/(f)$ are isomorphic. }
\end{compactitem}

\smallskip\noindent
As in the previous cases it follows that $t$ induces a prime element of the discrete valuation ring $R_C=A/(t^N-\pi)$ and that, with $K_C=Q(R_C)$, we have $K'_C:= B\otimes_A K_C \cong K_C[T] /(f_C)$, where $f_C$ is the polynomial in $K_C[T]$ induced by $f$. Then $f_C$ is of the form
\[
f_C=T^p  + a_{p-1}^C T^{p-1} + \cdots + a_0^C
\]
with $v_{R_C}(a_i^C)>0$ for every $i>0$ and such that
$a_0\in R_C^\times$ is not a $p$-th power in the residue field of $R_C$. It follows that $K'_C$ is a
discretely valued field ramified over $K_C$ with ramification index $e=1$.
\end{proof}

\section{Some valuation theory}\label{valsec}

Working over a general base scheme $S$, we first have to fix some notation.

\begin{definition}
We call an integral noetherian scheme $X$ {\em pure-dimensional} if $\dim  X= \dim \mathcal{O}_{X,x}$ for every closed point $x\in X$.
\end{definition}

\begin{remark}
Any integral scheme of finite type over a field or over a Dedekind domain with infinitely many prime ideals is pure-dimensional. A proper scheme over a pure-dimensional universally catenary scheme is pure-dimensional, see
 \cite{ega4}, IV, 5.6.5. The affine line ${\mathbb A}^1_{\Z_p}$ over the ring of $p$-adic integers gives an example of a regular scheme which is not pure-di\-men\-sio\-nal.
\end{remark}

Let from now on $S$ be an integral, pure-dimensional, separated and excellent base scheme.  We work in the category $\Sch(S)$ of separated schemes of finite type over $S$. In order to avoid the effect that open subschemes might have smaller (Krull-)dimension than the ambient scheme (e.g.\ $\Spec(\Q_p)\subset \Spec(\Z_p)$), we redefine the notion of dimension for schemes in $\Sch(S)$ as follows:

\smallskip
Let $X\in Sch(S)$ be integral and let $T$ be the closure of the image of $X$ in $S$. Then we put
\[
\dim_S X:= \mathrm{deg.tr.}(k(X)|k(T)) + \dim_{\textrm{Krull}} T.
\]
If the image of $X$ in $S$ contains a closed point of $T$,  then $\dim_S X=\dim_{\textrm{Krull}} X$ by \cite{ega4}, IV, 5.6.5.  This equality holds for arbitrary $X\in \Sch(S)$ if $S$ is of finite type over a field or over a Dedekind domain with infinitely many prime ideals.

\begin{definition}
A {\em function field over $S$} is the function field $K=k(X)$ of some integral scheme $X\in \Sch(S)$. We call $X$ a model of $K$ and put $\dim_S K=\dim_S X$.
\end{definition}

Every function field over $S$ admits a proper model, see~\cite{Lue} for a scheme theoretic proof. Let $K=k(X)$ be a function field over $S$ and let $T$ be the closure of the image of $X$ in $S$ (with reduced scheme structure). Then $T$ is an integral closed subscheme in $S$.

\begin{definition}
We denote by $\Val_S(K)$ the set of nonarchimedean valuations $v$ on $K$ such that $\O_v$ dominates $\O_{T,t}$ for some point $t\in T$.
For $v\in \Val_S(K)$ and for a given proper  $S$-model $X$ of $K$ we denote by $Z_v$ the center of $v$ on $X$, which is the (uniquely defined) integral closed subscheme of $X$ such that
$\O_v$ dominates $\O_{X,Z_v}$.
\end{definition}

Let $n\leq \dim_S K$ be a natural number. We consider $\Z^n$ as an ordered group with the lexicographic ordering.

\begin{definition}
We call an $S$-valuation $v: K^\times \twoheadrightarrow \Z^n$, i.e.~a valuation with ordered value group $\Z^n$, a discrete rank $n$ valuation on $K$. The residue field is denoted by $Kv$.
\end{definition}

The following proposition, due to Abhyankar, is well known. The inequality in the proposition is called the Abhyankar inequality.
\begin{proposition}\label{valuationtheory}
Let $K$ be a function field over $S$ and let $\bar X$ be a proper model of $K$. Let $v$ be a discrete rank $n$ valuation on $K$ with center $Z_v$  on $\bar X$. Then we have the inequality
\[
\mathrm{deg.tr.}(Kv|k(Z_v)) + \dim_S Z_v + n \leq \dim_S K\,.
\]
Suppose that equality holds. Then $Kv|k(Z_v)$ is a finitely generated field extension. Moreover, there exists a proper model $\bar X'$\/ of $K$ such that $\dim_SZ'_v=\dim_S K - n$ and  $Kv=k(Z'_v)$.
\end{proposition}

\begin{proof}
Let $z_v$ be the generic point of $Z_v$. Since $S$ is pure-dimensional, also $\bar X$ is  pure-dimensional and the Krull dimension of the ring
 $A:=\mathcal{O}_{\bar X,z_v}$ is equal to $\dim_S K - \dim_S Z_v$. So \cite[Th\'eor\`eme~9.2]{V} comprises our lemma except for the last statement.
In order to prove the latter, let $s_1,\ldots , s_m\in \mathcal{O}_v^\times$ be a finite family of elements which generate the residue field of $\mathcal{O}_v$ over
the residue field of the local ring $A$. Let $Spec(B)\subset \bar X$ be an open neighbourhood of $z_v$ and set $X':= Spec(B[s_1, \ldots , s_m])$. Then
$v$ has nontrivial a center in $X'$ and the function field of this center coincides with $Kv$. Finally, we let $X'\subset \bar X'\to \bar X$ be a compactification, which
exists by \cite{Lue}.
\end{proof}

We call $v$ {\em geometric} if equality holds in the Abhyankar inequality in Proposition~\ref{valuationtheory}. This notion does not depend
on the proper model we have chosen. If $v$ is geometric, then $Kv$ is a function field over $S$ of dimension $\dim_S K - n$. In case $n=\dim_S K$ the Abhyankar inequality is automatically an equality, so that every discrete rank $\dim_S K$ valuation is geometric and the center on every proper model is a closed point.

If $v$ is a valuation on $K$ and $w$ a valuation on $Kv$, then the ring
\[
\mathcal{O}_{v\circ w}:= \{ x \in \O_v \mid \bar x \in \O_w \}
\]
is a valuation ring and the associated valuation on $K$ is called the composite valuation $v\circ w$.

The following lemma is `folklore' but
we could not find a reference.

\begin{lemma} \label{valdeco}
Let $K$ be a function field over $S$.\smallskip

\begin{compactitem}
\item[\rm (i)]{ If\/ $v$ is a geometric discrete $S$-valuation of rank $n$ on $K$ and $w$ is a geometric discrete $S$-valuation of rank $m$ on $Kv$,
then $v\circ w$ is a geometric discrete rank $(n+m)$ valuation. }

\item[\rm (ii)]{ Each discrete geometric  $S$-valuation $v$ on $K$ of rank $n$ can be written in the form $v=v_1\circ \cdots \circ v_n$, where the $v_i$, $i=1,\ldots, n$, are geometric discrete rank $1$ valuations on $Kv_{i-1}$ (set $Kv_0:=K$).}
\end{compactitem}
\end{lemma}

\begin{proof}
(i) We have to show that $v\circ w$ is geometric. We use Proposition~\ref{valuationtheory} to find a proper model $\bar X$ of $K$ such that $Kv=k(Z_v)$,  where $Z_v$ is the center of~$v$ on $\bar X$. Then $Z_v$ is a proper model of $Kv$ and we denote the center of $w$ on $Z_v$ by $Z_w$. Note that $Z_w=Z_{v\circ w}$. The residue fields of $w$ and of $v\circ w$ are the same. We get the equality
\begin{eqnarray*}
n + ( \mathrm{deg.tr.}(Kw|k(Z_w)) + m ) & = & (\dim_S K - \dim_S Z_v) + (\dim_S Z_v - \dim_S Z_w)\\
 & = & \dim_S K - \dim_S Z_{v\circ w},
\end{eqnarray*}
which shows that $v\circ w$ is geometric.

(ii) Every discrete rank $n$ valuation $v$ can be uniquely decomposed into a chain $v=v_1\circ \cdots \circ v_n$ of discrete rank $1$ valuations. By Proposition~\ref{valuationtheory}, for any proper model $\bar X$ of $K$, the transcendence degree $\mathrm{deg.tr.}(Kv_1|k(Z_{v_1}))$ is finite. In a similar way as in the proof of Proposition~\ref{valuationtheory}, we find a proper model $\bar X$ of $K$ such that $\mathrm{deg.tr.}(Kv_1|k(Z_{v_1}))=0$, i.e.\ the extension is algebraic. Let $w$ be the restriction of $v_2\circ \cdots \circ v_n$ to $k(Z_{v_1})$. Then $Z_{v_1}$ is a proper model of $k(Z_{v_1})$ and the center $Z_w$ of $w$ on $Z_{v_1}$ is equal to $Z_v$. Furthermore, the residue fields of $v$ and of $v_2\circ \cdots v_n$ are the same, hence $Kv$ is an algebraic extension of the residue field  $Kw$ of $w$. Using the Abhyankar inequality for $v_1$ and for $w$, we obtain
\begin{eqnarray*}
\mathrm{deg.tr.}(Kv|k(Z_v)) + n & = &  1  + (\mathrm{deg.tr.}(Kw|k(Z_w)) + n-1 ) \\
 & \le &  (\dim_S K - \dim_S Z_{v_1}) + (\dim_S Z_{v_1} - \dim_S Z_w)\\
 & = & \dim_S K -\dim_S Z_v.
\end{eqnarray*}
Since $v$ is geometric, equality holds in both Abhyankar inequalities;  hence $v_1$ and $w$ are geometric. By Proposition~\ref{valuationtheory}, the extension $Kv_1|k(Z_{v_1})$ is finite, and so $Kv_1$ is a function field over $S$ and  $v_2\circ\cdots\circ v_n$ is a geometric discrete $S$-valuation of rank $n-1$ on this field. Now the result follows by induction.
\end{proof}

Let $v$ be a valuation on the field $K$, $L|K$  a finite Galois extension with Galois group $G$ and $w$ a valuation on $L$ which extends $v$.
Recall that the decomposition and the inertia group $G_w(L|K)$ and $T_w(L|K)$ of $w$ in $L|K$ are defined by
\[
G_w(L|K)=\{g\in G \mid gw=w\} \text{ and } T_w(L|K)=\ker \big(G_w(L|K) \to \Aut(Lw|Kv)\big),
\]
respectively. Furthermore, we have the ramification group
\[
V_w(L|K)= \{g\in G_w(L|K)\mid w(gx/x-1\big) > 0 \text{ for all } x \in L^\times \}.
\]
If $\hbox{char}(Kv)=0$, then $V_w(L|K)=1$, and if $p:=\hbox{char}(Kv)>0$, then $V_w(L|K)$ is the unique $p$-Sylow subgroup of $T_w(L|K)$.
If $w$ and $w'$ are different extensions of $v$, then the decomposition, inertia and ramification groups of $w$ and $w'$ are conjugate and we sometimes write $G_v(L|K)$, $T_v(L|K)$ and $V_v(L|K)$ for these groups if we don't care for conjugation.
There are several definitions of ramification for valuations used in the literature. We use the following:

\begin{definition}
Let $v$ be a valuation on a field $K$ and let $L|K$ be a finite Galois extension. We say that $v$ is {\em unramified} (resp.\ {\em tamely ramified}) in $L|K$ if $T_v(L|K)$ is trivial (resp.\ if $V_v(L|K)$ is trivial).
We say that $v$ is unramified (resp.\ tamely ramified) in a finite separable extension $L|K$ if it is unramified (resp.\ tamely ramified) in the Galois closure of $L|K$.
\end{definition}

\begin{remark}
Let $v$ be a valuation with residue characteristic $p$ on a field $K$ and let $L|K$ be a finite separable extension. Let $w_1,\ldots ,w_n$ be the different extensions of $v$ to $L$. We denote by $e(w_i|v)$ the ramification index of $w_i$ over $v$ and we put $f(w_i|v)=[Lw_i:Kv]$. By the results of \cite{End}, \S22, the valuation $v$ is unramified (resp.\ tamely ramified) in $L|K$ in the sense of the above definition if and only if the following conditions $(1)$--$(3)$ hold: \smallskip
\begin{compactitem}
\item[\rm (1)]  $e(w_i|v)=1$ (resp.\ $e(w_i|v)$ is prime to $p$) for $i=1,\ldots , n$,
\item[\rm (2)]  $Lw_i|Kv$ is separable for $i=1,\ldots , n$,
\item[\rm (3)]  $\sum_{i=1}^n e(w_i|v) f(w_i|v) = [L:K]$.
\end{compactitem}

\smallskip\noindent
If $v$ is a discrete rank $n$ valuation, then conditions (1) and (2) already imply (3).
\end{remark}

\begin{lemma}\label{val-purity}
Let $L|K$ be a finite separable extension of function fields over $S$ and let $d=\dim_S K$. Then the following are equivalent. \smallskip
\begin{compactitem}
\item[\rm (i)]{ $L|K$ is ramified at a (geometric) discrete rank $d$ valuation.}

\item[\rm (ii)]{ $L|K$ is ramified at a geometric discrete rank $1$ valuation.}
\end{compactitem}
\end{lemma}

\begin{proof}
(ii)$\Rightarrow$(i) Let $v$ be a ramified geometric discrete rank $1$ valuation. Choose a geometric discrete rank $(d-1)$ valuation $v_1$ on $Kv$ and put $w=v\circ v_1$; to find such a valuation $v_1$ one can use Parshin chains, see Section~\ref{tameness-section}. Then $w$ is a ramified discrete rank $d$ valuation.

(i)$\Rightarrow$(ii) We use induction on $d=\dim_S K$. For $d=1$ the lemma is trivial. In case $d=2$ and $v$ is a ramified discrete rank $2$ valuation of $K$ choose a proper normal model $X$ of $K$. Then consider the two-dimensional local ring $A=\O_{X,x}$ where $x$ is the generic point of the center of $v$ on $X$. After a modification we may assume that $A$ is regular, see Proposition~\ref{desing}. Since the normal closure of $A$ in $L$ is ramified over $A$, purity of the branch locus~\cite[X.3.4]{sga2} gives us a ramified divisor on $A$ which gives us a ramified geometric discrete rank $1$ valuation.
For $d>2$ write $v=v_1\circ w$ where $v_1$ is a discrete rank $1$ valuation, which is geometric according to Lemma~\ref{valdeco}~(ii). If $v_1$ ramifies in $L$ we are done. Otherwise
consider the finite product $L_v$ of the residue fields of all extensions of $v_1$ to $L$. Then, as $v$ ramifies in $L$ and as $v_1$ does not ramify in $L$,
one of the extensions of $w$ to $L_v$ ramifies (see~\cite{ZarSam}, Ch.~VI \S 11 Cor.~2 to Lem.~4).
Now we use the induction
assumption to find a geometric discrete valuation $v_2$ of $Kv_1$ of rank~$1$ which ramifies in $L_v$.  So $v_1\circ v_2$ is a geometric discrete rank~$2$ valuation on $K$ which ramifies in the extension $L|K$.
Using the induction assumption again, we find a geometric discrete rank $1$ valuation on $K$ which ramifies in $L|K$.
\end{proof}

\begin{lemma}\label{wild-purity}
Let $L$, $K$ and $d$ be as in Lemma~\ref{val-purity}.
Assume there exists a wildly ramified discrete rank $d$ valuation $v$ on $K$. Then there exists a wildly ramified geometric discrete rank $1$ valuation $w$ on $K$.
\end{lemma}

\begin{proof} Let $p$ be the characteristic of $Kv$. We may replace $L|K$ by its Galois closure. Since the ramification group of $v$ is $p$-group,
there exists an intermediate extension $K\subset K' \subset L' \subset L$ such that $L'|K'$ is Galois of prime order $p$ and such that an extension $v'$ of $v$ to $K'$ is wildly ramified in $L'|K'$. Choose a proper normal model $X$ of $K'$ over $S$.
By the Lemma~\ref{val-purity}, we find a discrete rank $1$ valuation $w'$ of $K'$ which ramifies in $L'|K'$. Choose a point $x\in Z_{w'}$ of codimension $2$ in $X$ and of residue characteristic $p$. Set $A=\mathcal{O}_{X, x}$. By Proposition~\ref{desing}, we may
assume that $A$ is regular and that the ramification locus of $L'|K'$ in $\Spec (A)$ has normal crossings. Then
Abhyankar's Lemma~\cite[XIII Proposition 5.2]{sga1} applied
to the strict henselization of $A$ shows that there is a ramified prime divisor on $\Spec (A)$ of function field characteristic $p$. Let $w''$ be
the discrete rank $1$ valuation corresponding to this ramified prime divisor. Then the restriction $w$ of $w''$ to $K$ is a wildly ramified geometric discrete rank $1$ valuation.
\end{proof}

\section{Curve-, Divisor- and Chain-Tameness} \label{tameness-section}

We start with the following applications of our Key Lemma~\ref{key-lemma}.

\begin{proposition}\label{unram-ext} Let $X$ be a regular, pure-dimensional, excellent scheme, $X'\subset X$ a dense open subscheme,  $Y'\to X'$  an \'{e}tale covering and $Y$  the normalization of\/ $X$ in $k(Y')$.
Suppose that for every curve $C$ on $X$ with $C'=C\cap X'\ne\varnothing$, the \'{e}tale covering $Y'\times_{X} \tilde C' \to   X'\times_X \tilde C'$ extends to an \'{e}tale covering of\/ $\tilde  C$. Then $Y\to X$ is \'{e}tale.
\end{proposition}

\begin{proof} Without loss of generality we can assume that $Y'\to X'$ is a Galois covering.
Assume $Y\to X$ were not \'{e}tale. We have to find a curve $C$ on $X$ with $C'=C\cap X'\ne\varnothing $ such that $Y'\times_{X'} \tilde C'\to \tilde C'$ is ramified along $\tilde C\sm \tilde C'$. By the purity of the branch locus
\cite[X.3.4]{sga2}, there exists a component $D$ of $X\sm X'$ of codimension one in~$X$ such that $Y\to X$
is ramified over the generic point of~$D$. Let $G$ be a cyclic subgroup of prime order of the inertia group of some point of~$Y$ which lies over the
generic point of~$D$. Let $Y'_G$ be the quotient of $Y'$ by the action of $G$. Consider the Galois covering
$Y' \to Y'_G$ of prime degree and let $Y_G$ be the normalization of $X$ in
$k(Y'_G)$. By considering the localization at any closed point of $Y_G$ lying over~$D$, Lemma~\ref{key-lemma} produces a curve $C_G$ on $Y_G$ with $C'_G=C_G \cap Y'_G\ne\varnothing$ such that $Y'\times \tilde C'_G \to \tilde C'_G$
is ramified along $\tilde C_G\sm \tilde C'_G$. Let
$C$ be the image of $C_G$ under the morphism $Y_G\to X$. Then $C$ is the curve we are looking for.
\end{proof}

\begin{proposition}\label{Groth-Wies}
Let $X$ be a normal, pure-dimensional, excellent scheme and let $D$ be a NCD on $X$ (in particular, $X$ is regular in a neighbourhood of $D$). Then an \'{e}tale covering  $Y'\to X':=X\sm D$ is tamely ramified along $D$ if and only if for each closed curve $C\subset X$ not contained in $D$ the base change $Y'\times_X  \tilde C' \to \tilde C'$ is tamely ramified along~$D_{\tilde C}$.
\end{proposition}

\begin{proof}
If $Y'\to X':=X\sm D$ is tamely ramified along $D$ and $C$ is a curve on $X$ we use Abhyankar's Lemma to show that
$Y'\times_X  \tilde C' \to \tilde C'$ is tamely ramified
along $D_{\tilde C}$. Let $x$ be a point of the intersection $D\cap C$ and $A$ the strict henselization of $\mathcal{O}_{X, x}$. We can replace $X$ by $\Spec (A)$. Then $|D|=V(\pi_1)\cup \cdots \cup
V(\pi_r) $ for some irreducible elements $\pi_1,\ldots , \pi_r\in A$ which are part of a parameter system of $A$ and $X'=\Spec(A')$ with
$A'= A_{\pi_1 \cdots \pi_r}$.
By Abhyankar's Lemma~\cite[XIII Proposition 5.2]{sga1}, we can assume without loss of generality that $Y'\to X'$ is a Galois covering of the form
\[
\Spec(A'[T_1,\ldots ,T_r]/(T_1^{n_1}-\pi_1,\ldots , T_r^{n_r}-\pi_r ) ) \to \Spec(A')
\]
for natural numbers $n_1,\ldots ,n_r$ which are prime to the residue characteristic $p$ of~$A$. So, in particular, $p$ does not divide $\deg(Y'|X')=n_1\cdots  n_r$. In this situation the degree of the Galois covering $Y'\times_X  \tilde C' \to \tilde C'$ divides  $\deg (Y'|X')$, so it must be tamely ramified along $D_{\tilde C}$.

For the reverse direction assume that $Y'\to X':=X\sm D$ is not tamely ramified along some irreducible component $D_0$ of $D$. We will construct
a curve $C$ on $X$ such that the base change $Y'\times_X  \tilde C' \to \tilde C'$ is not tamely ramified
along $D_0 \times_X \tilde C$.
We can replace $Y'$ by its Galois closure over $X'$. Let $G$ be a subgroup of prime order
of the wild ramification group of $Y'\to X'$ at some point over the generic point of $D_0$. Let $Y'_G$ be the quotient of $Y'$ under the action of $G$
and let $Y_G$ be the normalization of $X$ in $k(Y'_G)$. Then a localization of $Y'\to Y'_G\subset Y_G$ at a closed point of $Y_G$ lying over
$D_0$  satisfies the
assumptions of Lemma \ref{key-lemma}, so that we find a curve $C_G\subset Y_G$ with $C'_G=C_G\cap Y'_G \ne\varnothing$ such that $Y'\times \tilde C_G' \to \tilde C'_G$ is not
tamely ramified along $D_0\times_X \tilde C_G$. Finally, let the curve $C$ we are searching for be the image of $C_G$ under the morphism $Y_G\to X$.
\end{proof}

Let $S$ be an integral, pure-dimensional and excellent base scheme and $\Sch(S)$ the category of separated schemes of finite type over $S$.

\medskip
We call $C\in \Sch(S)$ a {\em curve} if $C$ is integral and $\dim_S C=1$. For a regular curve $C\in \Sch(S)$ there exists a unique regular curve $P(C)\in \Sch(S)$ which is proper over $S$ and contains $C$ as a dense open subscheme.  Note that $P(C)$ has Krull-dimension~$1$. So there is a unique notion of tameness for \'{e}tale coverings of regular curves in $\Sch(S)$. The next definition is motivated by Proposition~\ref{Groth-Wies}. It is the \lq maximal\rq\ definition of tameness which is stable under base change and extends the given one for curves.

\begin{definition}
Let $Y\to X$ be an \'{e}tale covering  in~$\Sch(S)$. We say that $Y\to X$ is {\em curve-tame} if for any morphism $C\to X$ in $\Sch(S)$ with $C$ a regular curve, the base change $Y\times_X  C \to  C$ is tamely ramified along $P(C)\sm C$.
\end{definition}

Below we will introduce versions of tameness which use Parshin chains on schemes. For the convenience of the reader we recall the definition of Parshin
chains and their connection with valuations.

\begin{definition}
Let $X\in \Sch(S)$ be a scheme. A finite family of points $P=(P_0, \ldots , P_r)$ on the scheme $X$ is called a {\em chain} if
\[
\overline{\{ P_0 \} } \subset \overline{\{ P_1 \} } \subset \cdots  \subset   \overline{\{ P_r \}}.
\]
The chain $P$ is called a {\em Parshin chain} if $\dim_S \overline{\{ P_i \} } = i$ for $0\le i\le r$.
\end{definition}

Assume we are given a scheme $X\in\Sch(S)$ and a Parshin chain $P=(P_0, \ldots , P_d)$ of length $d=\dim_S (X)$ on $X$. We say that a discrete valuation $v\in\Val_S(k(X))$
of rank $d$ {\em dominates} the Parshin chain $P$ if, for the unique decomposition $v=v_1\circ \cdots \circ v_d$ into discrete valuations of rank $1$ and $i=1,\ldots ,d$,  the valuation ring corresponding to $v_1\circ \cdots \circ v_i$ dominates $\mathcal{O}_{X,P_{d-i}}$. A Parshin chain $P$ is dominated by at least one  and by at most finitely many discrete valuations. If, for $i=1,\ldots , d$, $P_{i-1}$ is a regular point on $\overline{\{P_{i}\}}$, then there is exactly one such discrete valuation.

\bigskip
We introduce further definitions of tameness. Let $Y \to X$ be an \'{e}tale covering of connected normal schemes in $\Sch(S)$.
Then every $v\in \Val_S(k(X))$ with center in $X$ is unramified in $k(Y)|k(X)$.

\begin{definition}
Assume that $Y$ and $X$ are connected and normal. We say that $Y\to X$ is \smallskip
\begin{compactitem}
\item {\em divisor-tame} if for every normal compactification $\bar X$ of $X$ and every point $x\in \bar X \sm X$ with $\text{codim}_{\bar X} x=1$ the discrete rank one valuation $v_x$ on $k(X)$ associated with $x$ is tamely ramified in $k(Y)|k(X)$.
\item {\em valuation-tame} if every $v\in \Val_S(k(X))$ is tamely ramified in $k(Y)|k(X)$.
\item  {\em discrete-valuation-tame}, if  every discrete valuation $v \in \Val_S(k(X))$ of rank $d=\dim_S X$ is tamely ramified in $k(Y)|k(X)$.
\item  {\em chain-tame} if there exists a normal compactification $\bar X$ of $X$ such that each discrete valuation $v \in \Val_S(k(X))$ of rank $d=\dim_S X$ which dominates a Parshin-chain on $\bar X$ is tamely ramified in $k(Y)|k(X)$.
\item {\em weakly chain-tame} if there exists a normal compactification $\bar X$ of $X$ such that each discrete valuation $v \in \Val_S(k(X))$ of rank $d=\dim_S X$ which dominates a Parshin-chain $P=(P_0,\ldots, P_d)$ on $\bar X$ with $P_1,\ldots, P_d \in X$ is tamely ramified in $k(Y)|k(X)$.
\end{compactitem}

\smallskip\noindent
These definitions extend to the non-connected case by requiring the corresponding property for every connected component.
\end{definition}

Valuation-tameness obviously implies discrete-valuation-tameness and divisor-tameness. Discrete-valuation-tameness implies chain-tameness, which implies weak chain-tameness.

\begin{remark} \label{abghaengremark}
The question whether an \'{e}tale scheme morphism $Y\to X$ is tame  or not (in any of the above versions), depends on the category $\Sch(S)$ in which it is considered. For example, the \'{e}tale morphism $\Spec(\Z[\frac{1}{2},\sqrt{-1}])\to \Spec(\Z[\frac{1}{2}])$ is not tame in $\Sch(\Z)$, but is tame as a morphism in $\Sch(\Z[\frac{1}{2}])$. Another example is the following: any \'{e}tale covering $Y\to X$ of varieties over $\Q_p$ is tame when considered in $\Sch(\Q_p)$. This is in general not the case if we consider $Y\to X$ as a covering in $\Sch(\Z_p)$.
\end{remark}

Now we formulate our first main result.

\begin{theorem}\label{triangle-theo}
Let $S$ be an integral, excellent and pure-dimensional base scheme and let $Y\to X$ be an \'{e}tale covering of regular schemes in $\Sch(S)$. Then the following are equivalent:

\smallskip
\begin{compactitem}
\item[\rm (i)] $Y\to X$ is curve-tame.
\item[\rm (ii)] $Y\to X$ is divisor-tame.
\item[\rm (iii)] $Y\to X$ is discrete-valuation-tame.
\item[\rm (iv)] $Y\to X$ is chain-tame.
\item[\rm (v)] $Y\to X$ is weakly chain-tame.
\end{compactitem}

\smallskip\noindent
If\/ every intermediate field between $k(X)$ and the Galois closure of $k(Y)$ over $k(X)$ admits a regular proper model, then  then {\rm (i)--(v)} are equivalent to

\smallskip
\begin{compactitem}
\item[\rm (vi)] $Y\to X$ is valuation-tame.
\end{compactitem}

\smallskip\noindent
If there exists a regular compactification $\bar X$ of\/ $X$ such that $\bar X\sm X$ is a NCD, then {\rm (i)--(vi)} are equivalent and there
is  a further equivalent condition:

\smallskip
\begin{compactitem}
\item[\rm (vii)] $Y\to X$ is tamely ramified along $\bar X \sm X$.
\end{compactitem}
\end{theorem}

\begin{proof}
We start making the round. (i)$\Rightarrow$(ii) follows from the Key Lemma~\ref{key-lemma}. Let us show (ii)$\Rightarrow$(iii). Assume there exists a wildly ramified discrete valuation of rank $d=\dim_S X$. Using Lemma~\ref{val-purity}, we find a wildly ramified geometric discrete rank $1$ valuation $v$ on $k(X)$. Choose any normal compactification $\bar X$ of $X$.  As $Y\to X$ is \'{e}tale, the center of $v$ lies in $\bar X\sm X$. By \cite{Liu} \S 8, ex.\ 3.14, after blowing up $\bar X$ in centers contained in $\bar X \sm X$ and finally normalizing, we find a normal compactification $\bar X$ of $X$ such that $v$ is the valuation associated to a point $x\in \bar X\sm X$ of codimension $1$ in $\bar X$. This shows (ii)$\Rightarrow$(iii).

Furthermore, (iii)$\Rightarrow$(iv) and  (iv)$\Rightarrow$(v) are obvious. So assume (iv). Let $\bar X$ be a compactification of $X$ and  let $C \to X$ be a morphism in $\Sch(S)$, where $C$ is a regular curve. If the image of $C$ in $X$ is a closed point, then $Y\times_X C \to C$ extends to an \'{e}tale morphism of $P(C)$. So assume that the image $x$ of the generic point of $C$ in $X$ is one-dimensional and that the base change $Y\times_X C \to C$ is wildly ramified along $P(C)\sm C$. By the valuative criterion of properness, $C\to X$ extends to a morphism $P(C)\to \bar X$.
Let $v$ be a discrete rank $1$ valuation on $k(C)$ with center in $P(C)\sm C$ which is wildly ramified in $Y\times_X C \to C$. Let $w$ be the restriction of $v$ to $k(x)$ with respect to the inclusion $k(x)\subset k(C)$. Then $w$ is wildly ramified in $Y\otimes k(x)\to k(x)$. Let $P_0\in \bar X$ be the center of $w$. Since $x$ is a regular point on $X$, we find a Parshin chain $P_0,P_1,\ldots, P_d$ on $\bar X$ such that $P_1=x$ and $P_i$ is a regular point on $\overline{P_{i+1}}$ for $i=1,\ldots, d-1$. Then let $W$ be the geometric rank $d-1$-valuation associated to $P_1,\ldots, P_d$. Then $W\circ w$ is a wildly ramified discrete rank $d$-valuation. This shows (v)$\Rightarrow$(i).

Obviously, (vi) implies (iii). Now assume that (v) holds and that every intermediated field between $k(X)$ and the Galois closure of $k(Y)$ over $k(X)$ admits a regular proper model. In order to show (vi), we may replace $Y$ by its Galois closure over $X$.  Assume that there exists a valuation $v\in \Val_S(k(X))$ which is wildly ramified in $k(Y)|k(X)$. Let $p$ be the residue characteristic of $v$ and let $H\subset G(k(Y)|k(X))$ be a cyclic subgroup of order $p$ contained in the ramification group if $v$. We replace $X$ by the quotient scheme $Y/H$, i.e.\ we may assume that $Y|X$ is cyclic of order $p$. Let $X'$ be a regular proper model of $k(X)$ and let $x'$ be the center of $v$ in $X'$. Then $x'$ is ramified in $Y'\to X'$. By purity of the branch locus, we find a prime divisor $D\subset X'$ which contains $x'$ and is ramified in  $Y'\to X'$. Now choose a Parshin-chain from $x$ to $D$ and a discrete rank $d$ valuation which do\-mi\-nates this chain. This discrete rank $d$ valuation is wildly ramified in $k(Y)|k(X)$. Hence (iii) is violated.

Finally, assume that there exists a regular compactification $\bar X$ such that $\bar X \sm X$ is a normal crossing divisor. Then (vii) is equivalent to (i) by Proposition~\ref{Groth-Wies}. In the next section we will see (Theorem~\ref{numcrit}) that in this situation $Y\to X$ is numerically tame along $\bar X \sm X$, and therefore valuation-tame by Theorem~\ref{ntimpliesvaltame}. So (vii) implies (vi).
\end{proof}

\begin{remark}
If the scheme $X$ in Theorem~\ref{triangle-theo} is only assumed to be normal instead of regular, one still gets the implications
\[
\mathrm{(i)}\Rightarrow  \mathrm{(ii)}\Leftrightarrow \mathrm{(iii)} \Rightarrow   \mathrm{(iv)}   \Rightarrow   \mathrm{(v)}  .
\]
\end{remark}

\section{Numerical tameness}

Let $X$ be a connected normal scheme with function field $K=k(X)$ and let $L|K$ be a finite Galois extension with Galois group $G$. Let $Y$ be the normalization of $X$ in $L$ and let $y\in Y$ be a (not necessarily closed) point with image $x\in X$. Recall that the decomposition and the inertia group $G_y(Y|X)$ and $T_y(Y|X)$ of $y$ in $Y|X$ are defined by
\[
G_y(Y|X)=\{g\in G \mid gy=y\} \text{ and } T_y(Y|X)=\ker \big(G_y(Y|X) \to \Aut(k(y)|k(x))\big),
\]
respectively. If $y$ and $y'$ have the same image $x\in X$, then the decomposition and inertia groups of $y$ and $y'$ are conjugate and we sometimes write $G_x(Y|X)$ and $T_x(Y|X)$ for these groups if we don't care for conjugation.

\bigskip
Next we introduce the notion of numerical tameness.

\begin{definition}
Let $\bar X\in \Sch(S)$ be  normal connected and proper, and let $X \subset \bar X$ be a dense open subscheme. Let $Y\to X$ be an \'{e}tale Galois covering and let $\bar Y$ be the normalization of $\bar X$ in the function field $k(Y)$ of $Y$. We say that $Y \to X$ is {\em numerically tamely ramified along $\bar X \sm X$} if the order of the inertia group $T_x(\bar Y|\bar X) \subset G(\bar Y|\bar X)=G(Y|X)$ of each point $x\in \bar X \sm X$  is prime to the residue characteristic of $x$. An \'{e}tale covering $Y \to X$ is called numerically tamely ramified along $\bar X\sm X$ if it can be dominated by a Galois covering which is numerically tamely ramified along $\bar X\sm X$. This definition extends to the non-connected case by requiring numerical tame ramification for all connected components.
\end{definition}

\begin{remark}
If a point $x_0$ lies in the closure of another point $x_1$, then $T_{x_1}(Y|X)\subset T_{x_0}(Y|X)$. It therefore suffices to check numerical tameness on closed points.
\end{remark}

Numerical tameness is stable under base change in the following sense.

\begin{lemma} \label{numtameproperties} Let $\bar X\in \Sch(S)$ be  normal and proper, $X \subset \bar X$ a dense open subscheme and $Y\to X$  an \'{e}tale covering. Let $f: \bar X' \to \bar X$ be a proper morphism in $\Sch(S)$ with $\bar X'$ normal such that $X':=f^{-1}(X)$ is dense in $\bar X'$. If\/ $Y\to X$ is numerically tamely ramified along $\bar X\sm X$, then the base change $Y\times_X X' \to X'$ is numerically tamely ramified along $\bar X'\sm X'$.\smallskip
\end{lemma}
\begin{proof}
This follows since the inertia groups of the base change are subgroups of the inertia groups of $Y\to X$.
\end{proof}

\begin{theorem}\label{ntimpliesvaltame}
Let $X\in \Sch(S)$ be a regular scheme, $\bar X$ a normal compactification and $Y\to X$ an \'{e}tale covering. If\/  $Y\to X$ is numerically tamely ramified along $\bar X\sm X$, then it is valuation-tame. In particular, $Y\to X$ is curve-, divisor- and chain-tame.
\end{theorem}

\begin{proof} We may suppose that $Y\to X$ is Galois. Let $v$ be an $S$-valuation on $k(X)$ with residue characteristic $p>0$.  Let $x\in \bar X$ be a closed point the center of $v$. Then $x$ has residue characteristic $p$ and we have an inclusion $T_v(Y|X) \subset T_x(Y|X)$. Hence $T_v(Y|X)$ is of order prime to $p$ and $v$ is tamely ramified.
\end{proof}

\pagebreak
A partial converse in the case that $\bar X$ is regular is given by the following

\begin{theorem} \label{numcrit} Let $\bar X\in \Sch(S)$ be a regular, proper scheme, $X\subset \bar X$ a dense open subscheme and $Y\to X$ an \'{e}tale covering. Assume that one of the following conditions is satisfied: \smallskip
\begin{compactitem}
\item[\rm (a)]  $\bar X \sm X$ is a NCD  and $Y\to X$ is tamely ramified along $\bar X\sm X$, or \smallskip
\item[\rm (b)]  $Y\to X$ is curve-tame and can be dominated by a Galois covering with nilpotent Galois group.
\end{compactitem}

\smallskip\noindent
Then $Y \to X$ is numerically tamely ramified along $\bar X\sm X$.
\end{theorem}

\begin{remark}
The equivalence of numerical tameness and chain-tameness for nilpotent coverings has been shown in~\cite{S-tame}. Wiesend has given an incomplete
proof of the equivalence of numerical tameness and curve-tameness for nilpotent coverings in~\cite{W-tame}.
In the appendix we will give examples for curve-tame Galois coverings $Y \to X$ with non-nilpotent Galois group which are not numerically tamely ramified along $\bar X \sm X$ for some regular compactification $\bar X$. The first such example has been given by the referee of~\cite{W-tame}.
\end{remark}
\begin{proof}
Assume that $\bar X \sm X$ is a NCD and $Y\to X$ is tamely ramified along $\bar X\sm X$. By Abhyankar's Lemma \cite[XIII Proposition 5.2]{sga1}, the inertia group of every closed point $x\in \bar X$ has order prime to the residue characteristic of $x$. Hence $Y|X$ is numerically tamely ramified along $\bar X\sm X$.

Now assume that $\bar X \sm X$ is not necessarily a NCD but that $Y\to X$ can be dominated by a Galois covering with nilpotent group. Since a finite nilpotent group is the product of its Sylow subgroups, we may assume that $Y\to X$ is Galois and that $G=Gal(Y|X)$ is a finite $p$-group, where $p$ is some prime number. Assume that $Y|X$ were not numerically tamely ramified along $\bar X\sm X$. Then we find a closed point in $x_0\in \bar X\sm X$ with residue characteristic~$p$ which ramifies in $\bar Y|\bar X$.
Now we factor $Y\to X$ in the form
\[
Y=X_n \to X_{n_1} \to \cdots \to X_1\to X_0=X,
\]
such that $X_{i+1}\to X_{i}$ is Galois of degree~$p$ for $i=0,\ldots,n-1$. We denote by $\bar X_i$ the normalization of $\bar X$ in $k(X_i)$.  Let $a$, $0\leq a\leq n-1$, be the unique index such that $\bar X_a|\bar X$ is \'{e}tale over $x_0$ but $\bar X_{a+1}|\bar X_a$ is not etale
over the preimage $x_0$. Note that $\bar X_a$ is regular in a neighbourhood of the preimage of $x_0$, since there is a neighbourhood which is \'{e}tale over~$\bar X$. By purity, there exists a prime divisor on $\bar X_a$ which meets the preimage of $x_0$ in a closed point $x_a$ and which  ramifies in $\bar X_{a+1}|\bar X_a$. Applying the Key Lemma~\ref{key-lemma} to the localization of $\bar X_{a+1}|\bar X_a$ at $x_a$, we find a curve $\bar C\subset \bar X_a$ containing $x_a$ and with $C:=\bar C\cap X_a\neq\varnothing$, such that the base change $X_{a+1}\times_{X_a} \tilde C \to \tilde C$ ramifies along some point over $x_a$. We conclude that $X_{a+1}|X_a$ is not curve-tame. Hence $Y|X$ is not curve-tame. A contradiction.
\end{proof}

\section{Cohomological tameness}
Let $\bar X\in \Sch(S)$ be  normal connected and proper, and let $X \subset \bar X$ be a dense open subscheme. Let $Y\to X$ be an \'{e}tale Galois covering with group $G=\Gal(Y|X)$ and let $\bar Y$ be the normalization of $\bar X$ in the function field $k(Y)$ of $Y$. We denote the projection by $\pi: \bar Y \to \bar X$.

\begin{lemma} \label{cohlemma}
Using the notation as above let $\big(U_i=\Spec(A_i)\big)_{i\in I}$ be an affine Zariski-open covering of\/ $\bar X$ and put $\pi^{-1}(U_i)=:V_i=\Spec(B_i)$.   Then the following are equivalent.

\smallskip
\begin{compactitem}
\item[\rm (i)] For every closed point $x\in \bar X\sm X$ the semi-local ring $\O_{\bar Y, \pi^{-1}(x)}$ is a cohomologically trivial $G$-module. \smallskip
\item[\rm (ii)] For every point $x\in \bar X$ the semi-local ring $\O_{\bar Y, \pi^{-1}(x)}$ is a cohomologically trivial $G$-module. \smallskip
\item[\rm (iii)] $B_i$ is a cohomologically trivial $G$-module for all $i$. \smallskip
\end{compactitem}
\end{lemma}

\begin{proof}
For any prime ideal $\p\in A_i$, the localization $A_i \to (A_i)_\p$ is flat. Hence for every subgroup $H\subset G$ and all $i\in \Z$, we have
\[
\hat H^i(H, B_i \otimes_{A_i} (A_i)_\p) \cong \hat H^i(H, B_i) \otimes_{A_i} (A_i)_\p,
\]
where $\hat H^i(H,- )$ denotes the Tate cohomology of the finite group $H$. We conclude that $B_i$ is a cohomologically trivial $G$-module if and only if  $B_i \otimes_{A_i} (A_i)_\p$ is cohomologically trivial for every prime ideal, or even for every maximal ideal. In order to conclude the proof, it remains to show, that $B_i \otimes_{A_i} (A_i)_\p$ is a cohomologically trivial $G$-module for such $\p$ which are unramified $B_i$. To this end note that the natural homomorphism $(A_i)_\p \to (A_i)_\p^{sh}$ to the strict henselization is faithfully flat and that $B_i \otimes_{A_i} (A_i)_\p^{sh}$ is an induced $G$-module if $\p$ is unramified.
\end{proof}

\begin{definition}
We say that $Y \to X$ is {\em cohomologically tamely ramified along $\bar X \sm X$} if the equivalent conditions of Lemma~\ref{cohlemma} are satisfied. An \'{e}tale covering $Y \to X$ is called cohomologically tamely ramified along $\bar X\sm X$ if it can be dominated by a Galois covering which is cohomologically tamely ramified along $\bar X\sm X$. This definition extends to the non-connected case by requiring cohomological tame ramification for all connected components.
\end{definition}

\begin{theorem}\label{cohotame}
Let $\bar X\in \Sch(S)$ be  normal, connected and proper, and let $X \subset \bar X$ be a dense open subscheme. Then an \'{e}tale covering $Y \to X$ is  cohomologically tamely ramified along $\bar X \sm X$ if and only if it is numerically tamely ramified along $\bar X \sm X$.
\end{theorem}

\begin{proof} Let $\big(U_i=\Spec(A_i)\big)_{i\in I}$ be an affine Zariski-open covering of\/ $\bar X$ and put $\pi^{-1}(U_i)=:V_i=\Spec(B_i)$. Let $\tr_{B_i|A_i}: B_i \to A_i$ be the trace map $b\mapsto \sum_{g\in G} gb$. We omit the indices $i$ from the notation and call $B|A$ tamely ramified if for every ma\-xi\-mal ideal $\m\subset A$ the inertia group $T_\m(B|A)$ is of order prime to the characteristic of $A/\m$. The following claims are standard, cf.\ \cite{NSW}, Theorem 6.1.10.

\medskip \noindent
{\em Claim 1:} $\tr_{B|A}: B \to A$ is surjective if and only if $B|A$ is cohomologically tamely ramified.

\medskip\noindent
{\em Proof of claim 1.} For a maximal ideal $\m\subset A$, we denote the henselization of $A_\m$ by $A_\m^h$. Recall that $A_\m \to A_\m^h$  is faithfully flat. Hence $\tr_{B|A}: B \to A$ is surjective if and only if $\tr_{B|A} \otimes_A  A_\m^h: B \otimes_A  A_\m^h \to A_\m^h$ is surjective for all $\m\subset A$. Therefore we may assume that $A$ is local henselian with maximal ideal $\m$. Then $B$ is a finite product of local henselian rings and we may reduce to the case that $B$ is local henselian with maximal ideal $\m_B$. Let $\kappa$ and $\kappa'$ be the residue fields of $A$ and $B$, respectively. By Nakayama's Lemma, $\tr_{B|A}: B \to A$ is surjective if and only if the induced map $B/\m B \to A/\m$ is surjective. Since the trace sends $\m_B$ to $\m$, this map factors through a map $\tr: \kappa' \to \kappa$.  Let $p$ be the residue characteristic. By definition, the inertia group acts trivially on $\kappa'$. Hence $\tr$ is the zero map if $p\mid \# T_\m(B|A)$. If $e=\# T_\m(B|A)$ is prime to $p$, then $\kappa'|\kappa$ is separable and $\tr= e \cdot \tr_{\kappa'|\kappa}$ is surjective. This shows the claim.

\medskip \noindent
{\em Claim 2:} $B$ is a cohomologically trivial $G$-module if and only if $B|A$ is cohomologically tamely ramified.

\medskip\noindent
{\em Proof of claim 2.}
If $B$ is a cohomologically trivial $G$-module, then $\tr_{B|A}$ is surjective since $0=\hat H^0(G, B)= A/\tr_{B|A}(B)$. Hence $B|A$ is cohomologically tamely ramified by Claim 1. If $B|A$ is cohomologically tamely ramified and $H\subset G$ is a subgroup, then also $B|B^H$ is cohomologically tamely ramified. Hence $\tr_{B|B^H}$ is surjective by Claim 1, which implies $\hat H^0(H, B)=0$.

Next we prove that $H^1(H,B)=0$. Let $a(\sigma)\in B$ be a $1$-cocycle and let $x\in B$ be such that $\tr_{B|A}(x)=1$. Setting
$$
b:= \sum_{\sigma\in H}a(\sigma) \sigma x\,,
$$
we obtain for $\tau\in H$,
$$
\tau b=\sum_{\sigma\in H}\tau a(\sigma) (\tau\sigma x)=
\sum_{\sigma\in H}(a({\tau\sigma})-a(\tau))(\tau\sigma x) =
b-a(\tau)\tr_{B|A}(x)\,.
$$
Therefore $a(\tau) =(1-\tau)b$\,, hence $a(\tau)$ is a
$1$-coboundary. By \cite{NSW}, Proposition 1.8.4, we conclude that $B$ is cohomologically trivial. This shows Claim 2.

\medskip\noindent
Now we show the theorem. If $Y \to X$ is numerically tamely ramified, then $B_i$ is a cohomologically trivial $G$-module for all $i$ by Claim 2. Hence $Y\to X$ is cohomologically tamely ramified by Lemma~\ref{cohlemma}. The same arguments also show the reverse direction.
\end{proof}

 \section{Finiteness theorems}

In this section we will prove a tame version of a finiteness result due to Katz and Lang \cite{K-L} and will deduce finiteness theorems for the abelianized tame fundamental groups of arithmetic schemes, which had been previously shown in \cite{S-sing}.

\medskip
Let, as before,  $S$ be an integral, pure-dimensional and excellent base scheme.
Let $X\in \Sch(S)$ be regular and connected. We use the word tame for the equivalent notions of curve-, divisor-, discrete-valuation- and chain-tameness and  {\em tame covering} means a finite, \'{e}tale morphism which is tame. The tame coverings of $X$ satisfy the axioms of a Galois category (\cite{sga1}, V, 4). After choosing a geometric point $\bar x$ of $X$, we have the fibre functor $(Y\to X)\mapsto \Mor_X(\bar x, Y)$ from the category of tame coverings of $X$ to the category of sets, whose automorphism group is called the {\bf tame fundamental group} $\pi_1^t(X/S,\bar x)$. It classifies \'{e}tale coverings of $X$ which are tame when considered in $\Sch(S)$.  Denoting the usual \'{e}tale fundamental group by $\pi_1(X,\bar x)$, we have an obvious surjection
\[
\pi_1(X,\bar x) \twoheadrightarrow \pi_1^t(X/S,\bar x),
\]
which is an isomorphism if $X$ is proper over $S$. The fundamental groups to different base points are isomorphic, the isomorphism being canonical up to inner automorphisms. Therefore the abelianized fundamental groups
\[
\pi_1^\ab(X,\bar x)\ \text{ and }\ \pi_1^{t,\ab}(X/S,\bar x)
\]
are canonically independent of the base point $\bar x$, which we will exclude from notation. Note that the notion of (curve-) tameness is stable under base change. Therefore, given a morphism $f:Y\to X$ between regular, connected schemes in $\Sch(S)$, a geometric point $\bar y$ of $Y$ and its image $\bar x=f(\bar y)$ in $X$, we obtain an induced homomorphism
\[
\pi_1^t(Y/S,\bar y) \lang \pi_1^t(X/S,\bar x).
\]
Now assume that the function field $k$ of $S$ is absolutely finitely generated and let $p=\text{char}(k)$. Consider the kernel
\[
\ker(Y/X):= \ker \big (\pi_1^\ab(Y) \to \pi_1^\ab(X)\big)
\]
If $f: Y\to X$ is smooth and surjective, then a theorem of Katz and Lang (\cite{K-L}, Theorem 1) asserts that the prime-to-$p$ part of $\ker(Y/X)$ is finite. We will prove an analogous result for the tame kernel
\[
\ker^t_S(Y/X):= \ker \big (\pi_1^{t,\ab}(Y/S) \to \pi_1^{t,\ab}(X/S)\big),
\]
which does not exclude the $p$-part.

\begin{theorem}\label{Katz-Lang} Let
$S$ be an integral, pure-dimensional and excellent base scheme whose function field is absolutely finitely generated.  Let $f:Y\to X$ be a smooth, surjective morphism of connected regular schemes in $\Sch(S)$. Assume that the generic
fibre of $f$ is geometrically connected and that one of the following conditions {\rm (i)} and {\rm (ii)} is satisfied:

\medskip

\begin{compactitem}
\item[\rm (i)]  $X$ has a regular compactification $\bar X$ over $S$ such that the boundary $\bar X \sm X$ is a normal crossing divisor.\smallskip
\item[\rm (ii)] The generic fibre of $f$ has a rational point.
\end{compactitem}

\medskip\noindent
Then the group
\[
\ker_S^t (Y/X):=\ker \big( \pi_1^{t,\ab}(Y/S) \to  \pi_1^{t,\ab}(X/S) \big)
\]
is finite.
\end{theorem}

\begin{remark}
Condition~(i) is satisfied if $X$ has Krull-dimension less or equal to $2$; conjecturally, it is satisfied for an arbitrary $X$.
\end{remark}
In the proof of Theorem~\ref{Katz-Lang} we will need the following

\begin{lemma}\label{aufess}
Let $A$ be a discrete valuation ring and let $v$ be the associated discrete rank one valuation on the quotient field\/ $k$ of $A$. Let $K|k$ be a finitely generated, regular field extension and let $w\in \Val_{\Spec(A)}(K)$ be a geometric discrete rank one valuation on $K$ extending $v$.   Let $\pi\in k$ be a uniformizer for $v$ and let $k'|k$ be a finite Galois extension such that $w$ is unramified in $Kk'|K$. Then
\[
\# T_v(k'|k) \mid \; w(\pi) \cdot [(kv)_{Kw} :kv]_i,
\]
where $[(kv)_{Kw} :kv]_i$ is the index of inseparability of the relative algebraic closure $(kv)_{Kw}$ of $kv$ in $Kw$ (which is finite over $kv$ by Proposition~\ref{valuationtheory}).
\end{lemma}

\begin{proof} Without loss of generality we may assume that $A$ is henselian. Let $v'$ be the unique extension of $v$ to $k'$. Then $\# T_v(k'|k)=e_v(k'|k)[k'v':kv]_i$. Let $B\subset K$ be the valuation ring of $w$ and let $B'$ be the integral closure of $B$ in $Kk'$. Then $B'$ is a semi-local principal ideal domain. Let $\m$ be the maximal ideal of $B$ and let $\m_1'\ldots, \m_g'$ be the maximal ideals of $B'$. By assumption, $B\hookrightarrow B'$ is \'{e}tale, hence $\m B'= \m_1'\cdots \m_g'$. This implies
\[
\pi B' = \m^{w(\pi)} B'= (\m_1' \cdots \m_g')^{w(\pi)},
\]
But $\pi B'= (\pi')^{e_v(k'|k)}B'$, where $\pi'\in k'$ is a uniformizer of $v'$, and so we obtain $e_v(k'|k)\mid w(\pi)$.

Furthermore, for all $i$ the extension $(B'/\m'_i) | (B/\m)=Kw$ is separable by assumption. As $B'/\m'_i$ contains $k'v'$, we conclude that
\[
[k'v':kv]_i \mid [(kv)_{Kw} : kv].
\]
Summing up, the result follows.
\end{proof}

\begin{proof}[Proof of Theorem~\ref{Katz-Lang}] In the proof we will make frequent use of the following well-known fact: Let $X$ be a normal, connected locally noetherian scheme with generic point $\eta$ and function field $K$. Let $K^\sep$ be a separable closure of $K$ and let $\bar \eta$ be the corresponding geometric point of $X$. Then $\pi_1(X,\bar \eta)$ is a quotient of the Galois group $G(K^\sep|K)$ and the functor ``fibre over $\eta$''
\[
\{\text{connected \'{e}tale coverings of $X$}\} \lang \{\text{finite separable extensions of $K$}\}
\]
is fully faithful, with image those finite separable extensions $K'|K$ for which the normalization of $X$ in $K'$ is \'{e}tale over $X$. With the obvious modifications, the same also holds for the tame fundamental group.

\medskip\noindent
{\em First step:}
Reduction to $S=X=\Spec (K)$.

\smallskip\noindent
Let $\eta\in X$ be the generic point and let $K:=k(\eta)$ be its function field. To accomplish this step it clearly suffices to prove the following claim:

\smallskip
\noindent {\bf Claim:} {\it The natural map $\ker_\eta^t(Y_\eta / \eta ) \to \ker_S^t (Y/X) $ has finite cokernel.}

\smallskip\noindent
{\it Proof of the Claim:} Let $L$ be the function field of $Y$. We choose a separable closure $L^\sep$ of $L$ and denote the  separable closure of $K$ in $L$ by $K^{\sep}$. By assumption, $L$ and $K^\sep$ are linearly disjoint extensions of $K$. Therefore the natural map
$\pi_1^{t,\ab}(Y_{\eta}/\eta ) \to G (K^\sep|K)^{\ab}$ is surjective. The exact sequence
\[
1\to \pi_1^{t}(Y_{\bar{\eta}}/\bar{\eta}) \to \pi_1^{t}(Y_{\eta}/\eta )\to G (K^\sep|K) \to 1
\]
implies the exact sequence
\[
\pi_1^{t,\ab}(Y_{\bar{\eta}}/\bar{\eta})_{G(K^\sep|K)} \to \pi_1^{t,\ab}(Y_{\eta}/\eta )\to G (K^\sep|K)^\ab \to 0 \leqno (*)
\]
Therefore $\ker^t_\eta(Y_\eta/\eta)$ is a quotient of $\pi_1^{t,\ab}(Y_{\bar{\eta}}/\bar{\eta})_{G(K^\sep|K)}$.
Consider the commutative diagram
\[
\xymatrix{
\pi_1^{t,\ab}(Y_{\bar{\eta}}/\bar{\eta}) \ar[r] \ar@{=}[d] & \pi_1^{t,\ab}(Y_{\eta}/\eta ) \ar@{>>}[r] \ar@{>>}[d] & G (K^\sep|K)^{\ab}  \ar@{>>}[d]\\
\pi_1^{t,\ab}(Y_{\bar{\eta}}/\bar{\eta} )  \ar[r] &  \pi_1^{t,\ab}(Y/S ) \ar@{>>}[r] &  \pi_1^{t,\ab} ( X/S).
}
\]
The upper row is exact and the lower row is a complex. Denoting the cohomology of the lower row by $H$, we obtain an exact sequence
\[
\ker_\eta^t(Y_\eta / \eta ) \to \ker_S^t (Y/X) \to H.
\]
It therefore remains to show that $H$ is finite. In Galois terms, $H=G(K_2|K_1)$, where

\medskip
\begin{compactitem}
\item $K_1|K$ is the maximal abelian extension of $K$ such that the normalization of $X$ in $K_1$ is tame over $X$.
\item $K_2|K$ is the maximal abelian extension of $K$ such that the normalization of $Y$ in $LK_2$ is tame over $Y$.
\end{compactitem}

\medskip\noindent
We consider the cases (i) and (ii) separately.

\medskip\noindent
(i) Let $\bar X$ be a regular compactification of $X$ over $S$ such that $\bar X\sm X$ is a normal crossing divisor. Let $v_1,\ldots,v_n$ be the discrete rank one valuations of $K$ associated to the irreducible components of $\bar X\sm X$. Then, by Theorem~\ref{triangle-theo},  $H=G(K_2|K_1)$ is generated by the ramification groups of $v_1,\ldots,v_n$ in $G(K_2|K)$, and so it suffices to show that these groups are finite. Let $v$ be one of the $v_i$ and let $p$ be the residue characteristic. If $p=0$, there is nothing to prove, so assume $p>0$. Let $K_2(p)$ be the maximal $p$-subextension of $K$ in $K_2$. We have to show that the inertia group of $v$ in $K_2(p)|K$ is finite. Factor $Y\to \bar X$ into an open immersion $Y\subset \bar{Y}$ with dense image and normal $\bar Y$ and a proper morphism $\bar{Y} \to \bar{X}$, which is possible for example by~\cite{Lue}. We find a codimension one point on $\bar{Y}$ whose associated discrete valuation $w$ of $L$ extends $v$. Since the \'etale covering of $Y$ associated to $LK_2(p)|L$ is divisor-tame, we see that $w$ is unramified in $LK_2(p)|L$. Therefore $T_v(K_2(p)|K)$ is finite by Lemma~\ref{aufess}.

\medskip\noindent
(ii) As the generic fibre of $f$ has a rational point, there exists  a section $s: X_0\to Y$ to $f: Y\to X$ over some dense open subscheme $X_0\subset X$. If $X'\to X$ is an \'{e}tale covering such that $X'\times_X Y\to Y$ is tame, then, using the section $s$ and the base change property of (curve-)tame coverings, we see that $X'\times_X X_0\to X_0$ is tame. Taking the point of view of divisor-tameness, we conclude that $X'\to X$ is tame. Hence $H=0$ in this case.

\medskip\noindent
This completes the proof of the claim and of the first step. If $\text{char}(K)=0$, we are ready, because in this case
\[
\ker^t_\eta(Y/\eta)= \ker(Y/\eta)
\]
is finite by \cite{K-L}, Theorem 1. So we may assume that $K$ has characteristic~$p>0$.

\medskip\noindent
{\em Second step:} Reduction to the case that $f:Y\to X$ is an elementary fibration which admits a section $s: X\to Y$.

\smallskip
\noindent
Suppose we are in the situation achieved by the first step, i.e.~$X=S=\Spec(K)$. We claim that it suffices to prove the finiteness of $\ker_K^t(Y/K)$ after a finite base change $K\subset K'$. In fact, using the exact sequence $(*)$ of the first step, we see that the map $\ker_{K'}^t(Y\times_K K'/K')\to \ker_K^t(Y/K)$ is surjective, and this reduces us to show the finiteness of $\ker_{K'}^t(Y\times_K K'/K')$.

Let $Y'\to Y$ be \'{e}tale and let $K'$ be the algebraic closure of $K$ in $k(Y')$. We claim that we may replace $Y$ by $Y'$ and $K$ by $K'$. First note that $Y'\to Y$ factors through $Y'\to Y\times_K K'$, and so, by the last reduction, we may suppose that $K=K'$. Since the homomorphism $\pi_1^{\ab}(Y') \to \pi_1^{\ab}(Y)$ has finite cokernel, the same is true for $\pi_1^{t,\ab}(Y'/K) \to \pi_1^{t,\ab}(Y/K)$ and also for $\ker_K^t(Y'/K)\to \ker_K^t(Y/K)$. This shows the claim.

Recall that an elementary fibration is a complement of a finite \'{e}tale divisor in a proper and smooth relative curve. By \cite[Exp.~XI, 3.1]{sga4}, after replacing $Y$ by an \'{e}tale open, we can factor $f$ into a sequence of elementary fibrations
\[
Y=V_n \stackrel{f_n}{\lang} V_{n-1} \stackrel{f_{n-1}}{\lang} \cdots \stackrel{f_1}{\lang} V_0=X.
\]
Moreover, since a surjective, smooth morphism admits a section over an \'{e}tale open (\cite{ega4}, IV, 17.16.3), we can achieve the existence of sections $s_i:V_{i-1} \to V_{i}$ to $f_i: V_i \to V_{i-1}$ for all $i$.
Because the sequences
\[
0 \longrightarrow \ker_K^t(V_i/V_{i-1})  \longrightarrow \ker_K^t(V_i/K) \longrightarrow \ker_K^t(V_{i-1}/K) \longrightarrow 0
\]
are exact, it suffices to show the theorem for each elementary fibration in our sequence separately.

\medskip\noindent
{\it Last step:} Completion of the proof.

\medskip
\noindent
After the second step we are reduced to the case in which $f:Y\to X$ is an elementary fibration which has a section. A further application of the procedure in the first step allows us to assume that again $S=X=Spec(K)$, where $K$ is an absolutely finitely generated field $K$ of characteristic $p>0$. Let $\bar{Y}$ be the (uniquely defined) smooth compactification of the smooth curve $Y$ over $K$.  We have to show the finiteness of prime-to-$p$ and of the $p$-primary part of $\ker_K^t(Y/K)$.  The first is isomorphic to the prime-to-$p$ part of $\ker(Y/K)$. The latter is isomorphic to the $p$-part of $\ker(\bar{Y}/K)$. Therefore the proof of our theorem is finally reduced to~\cite[Theorems 1 and 2]{K-L}.
\end{proof}

Using Theorem~\ref{Katz-Lang}, one immediately deduces the following finiteness theorems for the tame fundamental group of arithmetic schemes, which had been shown previously in \cite{S-sing}.

\begin{theorem} \label{finite-tame-flat}
If\/  $X\in \Sch(\Z)$ is regular, connected and  flat over $\Spec(\Z)$, then the abelianized tame fundamental group $\pi_1^{t,\ab}(X/\Spec(\Z))$ is finite.
\end{theorem}

\begin{proof} For a dense open subscheme $X_0\subset X$, the homomorphism
\[
\pi_1^{t,\ab}(X_0/\Spec(\Z))  \lang \pi_1^{t,\ab}(X/\Spec(\Z))
\]
is surjective. Therefore, making $X$ smaller, we may assume that the structural morphism $X\to \Spec(\Z)$ is smooth.
Let $k$ be the algebraic closure of $\mathbb{Q}$ in $k(X)$. Then $X\to \Spec(\Z)$ factors through $\Spec(\O_k)$. Denoting the (open) image of $X$ in $\Spec(\O_k)$ by $U$, the morphism $X\to U$ is smooth and surjective with geometrically connected generic fibre. Now consider the exact sequence
\[
0\to \ker^t_{\Spec(\Z)} (X/U) \to \pi_1^{t,\ab}(X/\Spec(\Z)) \to \pi_1^{t,\ab}(U/\Spec(\Z))\to 0.
\]
The left hand group is finite by Theorem~\ref{Katz-Lang}. The group $\pi_1^{t,\ab}(U/\Spec(\Z))$ is finite by class field theory: it is the Galois group of the ray class field with modulus $\prod_{\p\notin U}\p$ of the number field $k$. This finishes the proof.
\end{proof}

If $X$ is a smooth, connected variety over a finite field $\F$, then $\pi_1^{t,\ab}(X/\Spec(\Z))=\pi_1^{t,\ab}(X/\Spec(\F))$, and we have the degree map
\[
\deg: \pi_1^{t,\ab}(X/\Spec(\F))\lang \pi_1(\F)\cong \hat \Z.
\]
The degree map has an open image, which  corresponds to the field of constants of $X$, i.e.\ the algebraic closure of $\F$ in $k(X)$.
In this situation, Theorem~\ref{Katz-Lang} reads as

\begin{theorem}\label{finite-tame-var}
Let $X$ be a smooth, connected variety over a finite field $\F$.
Then $\ker(\deg)$ is finite. In particular,
\[
\pi_1^{t,\ab}(X)\cong \hat \Z \oplus (\text{finite group}).
\]
\end{theorem}

\appendix

\section{Examples}

{\bf Example 1:} {\em We give an example which shows that the notion of tame ramification along a divisor on a regular scheme is not stable under base change if the divisor does not have normal crossings. This example is taken from \cite{S-tame}, Example 1.3.}

\medskip
Let $X=\Spec(\Z[T])$ be the affine line over $\Spec(\Z)$ and consider the divisor
\[
D= \text{div} (T+4) + \text{div}(T-4),
\]
which is not a normal crossing divisor. Let $K=\Q(T)$ be the function field of $X$ and
$U=X-D$. Put $f=(T+4)(T-4)=T^2-16$, $L=K(\sqrt{f})$ and consider the normalization $X_L$ of
$X$ in $L$. The ramification locus of $X_L \to X$ is either $D$ or $D \cup X_2$, where $X_2$
is the unique vertical divisor on $X$ over characteristic~$2$. Let us show that $X_L \to X$
is unramified at the generic point of $X_2$. This is equivalent to the statement that $L|K$
is unramified at the unique discrete valuation $v_2$ of $K$ which corresponds to the prime
ideal $2\Z[T]\subset \Z[T]$. Therefore it suffices to show that $f$ is a square in the
completion $K_2$ of $K$ with respect to $v_2$. Consider the polynomial $F(X)=X^2
-f=X^2-T^2+16$. We have  $F(T)\equiv 0 \bmod 16$ and the derivative $F'(T)= 2 T$ has the
exact $2$-valuation $1$. By the usual approximation process (cf.\ \cite{Ser-loc} 2.2.\ Theorem~1), we
see that $f$ has a square root in $K_2$.  Hence the ramification locus of $X_L \to X $ is
exactly $D$, and since $D$ is the sum of horizontal prime divisors, the morphism $U_L \to U$
is tamely ramified along $D$.

Now consider the closed subscheme $Y \subset X$ given by the equation \hbox{$T=0$}, so $Y
\cong \Spec(\Z)$. Then $D_Y= D \cap Y$ is the point  on $Y$ which corresponds to the prime
number $2$. Let $V= U \cap Y= Y -D_Y$. The base change $V'=U_L \times_U V \to V$ is the
normalization of $V\cong \Spec(\Z[\frac{1}{2}])$ in  $\Q(\sqrt{-1})$.  But $2$ is wildly
ramified in $\Q(\sqrt{-1})$, and so $V' \to V$ is not tamely ramified along  $D_Y$.

\bigskip\noindent
{\bf Example 2:} {\em We construct an example of a cyclic \'{e}tale covering of smooth varieties which is curve-tame but not numerically tame with respect to some normal but not regular compactification. Furthermore we give an example of a curve-tame Galois covering (with non-nilpotent Galois group) which is numerically tame with respect to some regular compactification but not numerically tame with respect to another regular compactification.}

\medskip
Let $k=\bar k$ be an algebraically closed field of characteristic $p>2$. Let $E$ be an ordinary elliptic curve over $k$ and let $y^2=x(x-1)(x-\lambda)$ be a Weierstra{\ss} equation for $E$. We obtain a cyclic degree~$2$ covering $\pi: E\to \Pe^1_k$ and we have
\[
E\sm \{\infty\}= \Spec\big(k[x,y]/y^2-x(x-1)(x-\lambda)\big) = \pi^{-1}\big(\A^1_k=\Spec(k[x])\big).
\]
Let $E'\to E$ be the unique cyclic \'{e}tale covering of degree~$p$. Because of uniqueness, $E'\to \Pe^1_k$ is a (tamely ramified) Galois covering of degree~$2p$. If $\Gal(E'|\Pe^1_k)$ would be abelian, we would obtain a cyclic covering of degree~$p$ of $\Pe^1_k$ which is tame, hence \'{e}tale because $p=\text{char}(k)$.  As $\Pe^1_k$ does not have nontrivial \'{e}tale coverings, the Galois group $\Gal(E'|\Pe^1_k)$ is nonabelian, hence isomorphic to the dihedral group $D_{2p}$ (which is not nilpotent).

Let $X=\pi^{-1}(\A^1_k \sm \{0,1,\lambda\})$ and let $X'$ be the preimage of $X$ in $E'$. Taking the product with $\A^1_k$, we obtain an \'{e}tale Galois covering with Galois group $D_{2p}$
\[
\phi: X'\times \A^1_k \xrightarrow{\text{\'{e}tale degree }p}  X'\times \A^1_k \xrightarrow{\text{\'{e}tale degree }2} (\A^1_k \sm \{ 0,1,\lambda\}) \times \A^1_k.
\]
The covering $\phi$ is numerically tame with respect to the compactification
\[
\bar \phi: E'\times \Pe^1_k \xrightarrow{\text{\'{e}tale degree }p} E\times \Pe^1_k \xrightarrow{\text{degree }2} \Pe^1_k \times \Pe^1_k,
\]
in particular, $\phi$ is curve-tame.
We will show that $\phi$ is not numerically tame with respect to another regular compactification of $(\A^1_k \sm \{ 0,1,\lambda\}) \times \A^1_k$:

Consider the embedding
\[
(\A^1_k \sm \{ 0,1,\lambda\})\times \A^1_k  \hookrightarrow \Pe^2_k,\ (x,t)\mapsto (x:t:1),
\]
and denote the normalizations of $\Pe^2_k$ in the function fields of $X\times \A^1_k$ and $X'\times \A^1_k$ by $Y$ and $Y'$, respectively.

We claim that the curve-tame cyclic covering $X'\times \A^1_k \to X\times \A^1_k$ is not numerically tame with respect to the normal compactification $X\times \A^1_k \hookrightarrow Y$. From this it follows that $\phi: X'\times \A^1_k \to (\A^1_k \sm \{ 0,1,\lambda\})\times \A^1_k$ is not numerically tame with respect to the regular compactification $(\A^1_k \sm \{ 0,1,\lambda\})\times \A^1_k \hookrightarrow \Pe^2_k$.

\medskip
In order to show the claim, it suffices to prove that there is exactly one point in $Y'$ over $P=(0:1:0)\in \Pe^2_k$. This can be seen as follows: $P$ corresponds to the maximal ideal $\m=(t^{-1}x,t^{-1})\subset k[t^{-1}x,t^{-1}]$. This ring is contained in $k[x,t^{-1}]$ and $\m$ is the preimage of the principal prime ideal $\p=(t^{-1})\subset k[x,t^{-1}]$. It therefore suffices to show that there is exactly one prime ideal above $\p$ in the integral closure of $k[x,t^{-1}]$ in $K'$. But this is easily seen: setting $E'\sm {\pi'}^{-1}(\infty)=\Spec(A')$, where $\pi': E'\to \Pe^1_k$ is the projection, the integral closure of $k[x,t^{-1}]$ in $K'$ is just $A'[t^{-1}]$.

\bigskip\noindent
{\bf Example 3:} {\em We construct a similar example as in Example 2, but with arithmetic surfaces instead of varieties. This example assumes some familiarity with S.~Saito's class field theory for curves over local fields, see \cite{Saito-local}.}

\bigskip
Let $p\neq 2$ be a prime number, $k|\Q_p$ a $p$-adic field, $E\to k$ an elliptic curve and  $E\to \Pe^1_k$ the degree 2 covering defined by a Weierstra{\ss} model. Then the normalization $\mathcal{W}$ of $\Pe^1_{\O_k}$ in $k(E)$ is a normal model of $E$ over $\Spec(\O_k)$.  Let $\mathcal{E}\to \mathcal{W}$ be a minimal resolution such that $\mathcal{E}$ is a regular model of $E$ and such that the reduced special fibre $(\mathcal{E}_s)_\red$ of $\mathcal{E}$ is a normal crossing divisor. Recall that the dual graph $\varGamma$ of $(\mathcal{E}_s)_\red$ is defined as follows:
the vertices correspond to the irreducible components and an edge connecting two vertices corresponds
to an intersection point of the two components.

Let $E'\to E$ be the maximal \'{e}tale elementary-abelian $p$-covering of $E$ in which all closed points of $E$ split completely. By \cite{Saito-local}, Proposition 2.2, it extends to an \'{e}tale covering $\mathcal{E'}\to \mathcal{E}$, hence $E'\to E$ is numerically tame with respect to the regular compactification $E\subset \mathcal{E}$. Furthermore, by loc.cit.\ Proposition 2.3 and Theorem 2.4, the covering $E'\to E$ is finite and there is a natural isomorphism $G(E'|E)\cong \pi_1^\ab(\varGamma)/p$.

Now assume that there exists a (singular) rational point $P$ in the special fibre $\mathcal{W}_s$ of $\mathcal{W}$ such that the subgraph $\varGamma_P$ of $\varGamma$ given by the components over $P$ is not contractible. Then not every elementary-abelian $p$-covering of $\varGamma$ splits over $\varGamma_P$, hence $\mathcal{E'}\to\mathcal{E}$ cannot come by base change from an \'{e}tale covering of $\mathcal{W'}\to\mathcal{W}$. We conclude that $E'\to E$ is not numerically tame with respect to the normal compactification $E\subset \mathcal{W}$. Furthermore, considering the composite map $E'\to E\to \Pe^1_k$, we obtain an example of a curve-tame but not numerically tame covering of some open subscheme $U\subset \Pe^1_k$.
This example can be ``lifted'' to obtain a similar example over  the ring of integers of a number field.

\vskip1.5cm
NWF I-Mathematik, Universit\"{a}t Regensburg, D-93040 Regensburg, Deutschland. {\it E-mail address:} {\tt moritz.kerz@mathematik.uni-regensburg.de}

\medskip
NWF I-Mathematik, Universit\"{a}t Regensburg, D-93040 Regensburg, Deutschland. {\it E-mail address:} {\tt alexander.schmidt@mathematik.uni-regensburg.de}


\begin{thebibliography}{CEPT}

\bibitem[Ab]{Ab} Abbes, A. {\em The Grothendieck-Ogg-Shafarevich formula for arithmetic
surfaces}.  J. Algebraic Geom. {\bf 9}  (2000),  no. 3, 529--576.

\bibitem[Bo]{Bour-comm}
Bourbaki, N. {\em
Commutative algebra}. Elements of Mathematics. Springer-Verlag, Berlin, 1998.

\bibitem[CE]{CE} Chinburg, T.; Erez, B. {\em Equivariant Euler-Poincar\'{e} characteristics and tameness}. Journ\'{e}es arithm\'{e}tiques 1991. Soci\'{e}t\'{e} Math.\ de France, Ast\'{e}risque. {\bf 209} (1992), 179--194.

\bibitem[CEPT]{CEPT} Chinburg, T.; Erez, B.; Pappas, G.; Taylor, M. J. {\em Tame actions of group schemes: integrals and slices}.  Duke Math. J.  {\bf 82}  (1996),  no. 2, 269--308.

\bibitem[EGA4]{ega4} Grothendieck, A. {\em \'{E}l\'{e}ments de g\'{e}om\'{e}trie alg\'{e}brique. IV. \'{E}tude locale des sch\'{e}mas et des morphismes de sch\'{e}mas (EGA 4)}.  Inst. Hautes \'{E}tudes Sci. Publ. Math. I. No. 20 1964, II. No. 24 1965, No. 28 1966,  IV. No. 32 1967.

\bibitem[End]{End} Endler, O. {\em Valuation theory}.  Universitext. Springer-Verlag, New York-Heidelberg, 1972.

\bibitem[GM]{G-M} Grothendieck, A.; Murre, J. P. {\em
The tame fundamental group of a formal neighbourhood of a divisor with normal crossings on a scheme}.
Lecture Notes in Math., Vol. 208, Springer-Verlag, Berlin-New York, 1971.


\bibitem[KL]{K-L} Katz, N.; Lang, S. {\em Finiteness theorems in geometric classfield theory. With an appendix by Kenneth A. Ribet}. Enseign. Math. (2) {\bf 27} (1981), no. 3-4, 285--319.


\bibitem[Lip]{Lip} Lipman, J. {\em Desingularization of two-dimensional schemes}. Ann. Math. {\bf 107} (1978), 151--207.

\bibitem[Liu]{Liu} Liu, Q. {\em Algebraic geometry and arithmetic curves}. Oxford Graduate Texts in Mathematics, 6. Oxford Science Publications. Oxford University Press, Oxford, 2002.

\bibitem[Lue]{Lue} L\"{u}tkebohmert, W. {\em On compactification of schemes}. Manuscripta Math.\ {\bf 80} (1993), 95--111.

\bibitem[NSW]{NSW} Neukirch, J.; Schmidt, A.; Wingberg, K. {\em Cohomology of number fields. 2nd ed.} Grundlehren der Math. Wiss., Bd. 323. Springer-Verlag, Berlin, 2008.

\bibitem[Sa]{Saito-local} Saito, S. {\em Class field theory for curves over local fields}.
J. Number Theory {\bf  21} (1985), no. 1, 44--80.

\bibitem[Sc1]{S-tame} Schmidt, A. {\em Tame coverings of arithmetic schemes}.  Math. Ann. {\bf 322}  (2002),  no. 1, 1--18.

\bibitem[Sc2]{S-sing} Schmidt, A. {\em Singular homology of arithmetic schemes}. Algebra \& Number Theory {\bf 1} (2007), no. 2, 183--222.

\bibitem[Se]{Ser-loc} Serre, J.-P. {\em Local fields}. Graduate Texts in Math., Vol. 67, Springer-Verlag, New York-Berlin, 1979.


\bibitem[SGA1]{sga1} Grothendieck, A. {\em Rev\^{e}tements \'{e}tales et groupe fondamental (SGA 1)}. Lecture Notes in Math., 224, Springer-Verlag, Berlin-New York, 1971.

\bibitem[SGA2]{sga2} Grothendieck, A. {\em Cohomologie locale des faisceaux coh\'{e}rents et th\'{e}or\`{e}mes de Lefschetz locaux et globaux (SGA 2)}. Advanced Studies in Pure Math., Vol. 2. North-Holland Publishing Co., Amsterdam; Masson \& Cie, \'{E}diteur, Paris, 1968.

\bibitem[SGA4]{sga4} Artin, M.; Grothendieck, A. and et  Verdier, J. L. {\em Th\'{e}orie des topos et cohomologie \'{e}tale des sch\'{e}mas (SGA 4)}. Lecture Notes in Math., 269, 270 and 305, 1972/3.


\bibitem[Sha]{saf} Shafarevich, I. G. {\em Lectures on minimal models and birational transformations of two dimensional schemes}. Tata Inst. Fundam. Res. 37, Bombay 1966.


\bibitem[Va]{V} Vaqui\'e, M. {\em Valuations}.  Resolution of singularities (Obergurgl, 1997), 539--590,
Progr. Math., 181, Birkh\"{a}user, Basel, 2000.

\bibitem[Wi]{W-tame} Wiesend, G. {\em Tamely ramified covers of varieties and arithmetic schemes}. Forum Math. {\bf 20} (2008), no. 3, 515--522.

\bibitem[ZS]{ZarSam} Zariski, O.; Samuel, P.
{\em Commutative algebra. Vol. II. }
Graduate Texts in Mathematics, Vol. 29. Springer-Verlag, New York-Heidelberg, 1975.

\end{thebibliography}
\end{document}